\documentclass[11pt,reqno]{amsart}

\usepackage{cite,a4wide,amsmath,amssymb,dsfont,epsfig,graphicx,subfig,verbatim,multicol,enumitem,lipsum,amsaddr}
\usepackage[hang,flushmargin]{footmisc}

\theoremstyle{definition}

\newtheorem{theorem}{Theorem}
\newtheorem{proposition}[theorem]{Proposition}
\newtheorem{lemma}[theorem]{Lemma}
\newtheorem{corollary}[theorem]{Corollary}

\let\oldproofname=\proofname
\renewcommand{\proofname}{\rm\bf{\oldproofname}}

\newenvironment{proofFP}{\noindent\textbf{Proof of Theorem \ref{thm:FP}.}}{\hfill$\square$}

\newenvironment{proof2cycles}{\noindent\textbf{Proof of Theorem \ref{thm:2cycles}.}}{\hfill$\square$}

\newcommand{\Z}{\mathbb{Z}}

\DeclareMathOperator{\Fix}{Fix}
\DeclareMathOperator{\Cyc}{Cyc}

\usepackage{pgf,tikz}
\usepackage{tkz-fct}
\usepackage{pgfplots}
\usetikzlibrary{arrows,trees}

\begin{document}
\title{The dynamics of one-dimensional quasi-affine maps}
\author{\small Jonathan Hoseana}
\address{\normalfont\small Center for Mathematics and Society, Department of Mathematics, Parahyangan Catholic University, Bandung 40141, Indonesia}
\email{j.hoseana@unpar.ac.id}
\date{}

\begin{abstract}
We study the dynamics of the one-dimensional quasi-affine map $x\mapsto \left\lfloor \lambda x +\mu \right\rfloor$, providing a complete description of the map's periodic points, and of the limit points of every $x\in\mathbb{R}$ under the map, for all real parameter values. Specifically, we establish the existence of regions of parameter values for which the map possesses $n$ fixed points for all $n\in\mathbb{N}_0\cup \{\infty\}$, an explicit formula for the number of 2-cycles possessed by the map, and the $\omega$-limit set of any $x\in\mathbb{R}$ under the map, which, depending on the parameter values, is either a singleton of a fixed point, a 2-cycle, $\{-\infty,\infty\}$, $\{\infty\}$, or $\{-\infty\}$.

\smallskip\noindent
\textsc{Keywords.} quasi-affine; floor function; periodic point; limit point

\smallskip\noindent\textsc{2020 MSC subject classification.} 37E99; 37C25
\end{abstract}

\maketitle

\section{Introduction}

When the orbits of a dynamical system are generated and visualised using a computer, there is always a question of whether the results provide an accurate description of the system's actual orbits, despite the computer's finite capability. In \cite{LiCorless}, for instance, the periodic orbits of the maps $x\mapsto \frac{x^2-1}{2x}$ and $x\mapsto \frac{x^3-3x}{3x^2-1}$ were computed with various levels of precision, exposing the unavoidable departure from periodicity due to round-off errors. Similarly, in \cite{HoseanaMSc}, the orbit of the so-called mean-median map \cite{ChamberlandMartelli} with initial sequence $\left(0,\frac{\sqrt{5}-1}{2},1\right)$ was computed, first with 10 and subsequently with 20 significant digits accuracy (Figure \ref{fig:mmm}), revealing remarkable discrepancy, which also originates from round-off errors. 

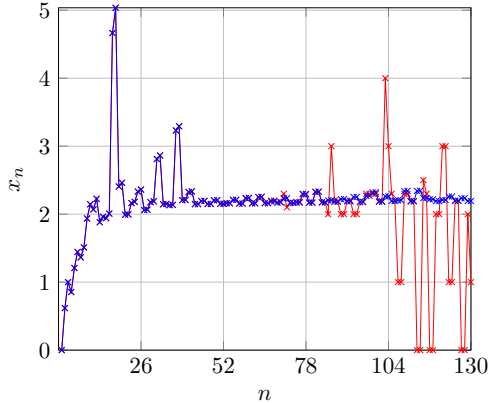
\begin{figure}[h!]
\centering\scalebox{0.8}{
\begin{tikzpicture}
\begin{axis}[
	xmin=0,
	xmax=130,
	ymin=0,
	ymax=5.0353932797412139,
	xtick={26,52,78,104,130},
	ytick={0,1,2,3,4,5},
	samples=100,
	xlabel=$n$,
	ylabel={$x_n$},
    grid=major,
    ylabel near ticks,
]
\addplot [red,mark=x] coordinates {
(1, 0.) (2, .6180339880) (3, 1.) (4, .854101966) (5, 1.208203931) (6, 1.444271908) (7, 1.364745086) (8, 1.51064312) (9, 1.93691769) (10, 2.14512162) (11, 2.06918028) (12, 2.22572144) (13, 1.8816694) (14, 1.9611962) (15, 1.9420560) (16, 2.0084272) (17, 4.6643667) (18, 5.035393) (19, 2.4065279) (20, 2.4617762) (21, 1.990870) (22, 1.996008) (23, 2.162169) (24, 2.181309) (25, 2.33212) (26, 2.36179) (27, 2.060237) (28, 2.065376) (29, 2.17608) (30, 2.18850) (31, 2.81148) (32, 2.86329) (33, 2.145019) (34, 2.150158) (35, 2.13196) (36, 2.13576) (37, 3.231) (38, 3.293) (39, 2.2062) (40, 2.2100) (41, 2.326) (42, 2.335) (43, 2.1472) (44, 2.1474) (45, 2.193) (46, 2.195) (47, 2.1496) (48, 2.1497) (49, 2.203) (50, 2.207) (51, 2.1522) (52, 2.1524) (53, 2.162) (54, 2.162) (55, 2.21) (56, 2.21) (57, 2.1551) (58, 2.1552) (59, 2.236) (60, 2.239) (61, 2.1583) (62, 2.1583) (63, 2.249) (64, 2.254) (65, 2.161) (66, 2.162) (67, 2.2) (68, 2.2) (69, 2.17) (70, 2.18) (71, 2.3) (72, 2.1) (73, 2.165) (74, 2.165) (75, 2.18) (76, 2.18) (77, 2.30) (78, 2.29) (79, 2.169) (80, 2.169) (81, 2.33) (82, 2.33) (83, 2.173) (84, 2.174) (85, 2.) (86, 3.) (87, 2.20) (88, 2.19) (89, 2.) (90, 2.) (91, 2.20) (92, 2.20) (93, 2.) (94, 2.) (95, 2.178) (96, 2.178) (97, 2.30) (98, 2.30) (99, 2.3) (100, 2.3) (101, 2.184) (102, 2.183) (103, 4.) (104, 3.) (105, 2.3) (106, 2.2) (107, 1.) (108, 1.) (109, 2.3) (110, 2.3) (111, 2.189) (112, 2.190) (113, 0.) (114, 0.) (115, 2.5) (116, 2.3) (117, 0.) (118, 0.) (119, 2.) (120, 2.) (121, 3.) (122, 3.) (123, 1.) (124, 1.) (125, 2.196) (126, 2.195) (127, 0.) (128, 0.) (129, 2.) (130, 1.)
};
\addplot [blue,mark=x] coordinates {
(1, 0.) (2, .61803398874989484820) (3, 1.) (4, .8541019662496845446) (5, 1.2082039324993690892) (6, 1.4442719099991587856) (7, 1.3647450843757886385) (8, 1.510643118126104094) (9, 1.936917696247160901) (10, 2.145121628746529991) (11, 2.069180267819676610) (12, 2.225721419696096160) (13, 1.88166945092769459) (14, 1.96119627655106474) (15, 1.94205597095124860) (16, 2.00842717907819391) (17, 4.66436694693962335) (18, 5.0353932797412139) (19, 2.40652778146262451) (20, 2.46177602678209082) (21, 1.9908695806400817) (22, 1.9960078553441694) (23, 2.1621694853491343) (24, 2.1813097909489504) (25, 2.332112577663777) (26, 2.361785881752794) (27, 2.0602362891452656) (28, 2.0653745638493533) (29, 2.176088049487525) (30, 2.188507373221549) (31, 2.811468385117805) (32, 2.863277495184877) (33, 2.1450178217627126) (34, 2.1501560964668003) (35, 2.131974383330011) (36, 2.135780087300334) (37, 3.2308714047609) (38, 3.2936655202712) (39, 2.20618561075132) (40, 2.20999131472164) (41, 2.3251536437791) (42, 2.3343913782415) (43, 2.14724967191479) (44, 2.14735347889860) (45, 2.1930026000323) (46, 2.1951306432006) (47, 2.14968913603449) (48, 2.14979294301831) (49, 2.2045770787279) (50, 2.2069127358638) (51, 2.15233621412184) (52, 2.15244002110566) (53, 2.1594165094032) (54, 2.1597796628517) (55, 2.210109331980) (56, 2.212289449635) (57, 2.15529471316063) (58, 2.15539852014445) (59, 2.2366534367275) (60, 2.2395081287825) (61, 2.15846082616706) (62, 2.15856463315088) (63, 2.2518611598568) (64, 2.2549234658794) (65, 2.1618345531411) (66, 2.1619383601249) (67, 2.18710248761) (68, 2.18795436386) (69, 2.171945303376) (70, 2.172308456825) (71, 2.23272826813) (72, 2.23478315841) (73, 2.1656235080505) (74, 2.1657273150343) (75, 2.170605556032) (76, 2.170836681256) (77, 2.295149359350) (78, 2.298603382052) (79, 2.1697238839112) (80, 2.1698276908951) (81, 2.327588354552) (82, 2.331584923429) (83, 2.1740318737397) (84, 2.1741356807235) (85, 2.2028869592) (86, 2.2036648244) (87, 2.180659503284) (88, 2.180890628508) (89, 2.2201703656) (90, 2.2212789877) (91, 2.188468785282) (92, 2.188831938731) (93, 2.2524473434) (94, 2.2541707603) (95, 2.1789627054710) (96, 2.1790665124548) (97, 2.268825565779) (98, 2.270777934543) (99, 2.31838352067) (100, 2.32125817665) (101, 2.1842049581538) (102, 2.1843087651376) (103, 2.2611055402) (104, 2.2626985310) (105, 2.19279357755) (106, 2.19302470278) (107, 2.2033158191) (108, 2.2037349815) (109, 2.33909640361) (110, 2.34199157082) (111, 2.1899662457556) (112, 2.1900700527395) (113, 2.342154085) (114, 2.344947807) (115, 2.23608537212) (116, 2.23693724837) (117, 2.218048017) (118, 2.218562439) (119, 2.1907647677) (120, 2.1908033556) (121, 2.2081435865) (122, 2.2084681520) (123, 2.2585918208) (124, 2.2597261278) (125, 2.1964541822442) (126, 2.1965579892280) (127, 2.23418445) (128, 2.23487917) (129, 2.1932536898) (130, 2.1932922777)
};
\end{axis}
\end{tikzpicture}}
\caption{\label{fig:mmm}Plot of the first 130 points in the orbit $\left(x_n\right)_{n=1}^\infty$ of the mean-median map with initial sequence $\left(0,\frac{\sqrt{5}-1}{2},1\right)$, with 10 significant digits accuracy (red) and with 20 significant digits accuracy (blue).}
\end{figure}
To study such phenomena in a formal setting, dynamicists have devised the concept of \textit{spatial discretisation} \cite{BeckRoepstorff,Blank,DiamondKloeden,DiamondKloedenPokrovskii1,DiamondKloedenPokrovskii2,DiamondSuzukiKloedenPokrovskii,Guiheneuf}, whereby the orbital behaviour of a map $F:X\to X$ is compared to that of a variant $f:=D\circ F$, where $D:X\to X$ is a map ---typically neither injective nor surjective--- specified to manifest the round-off \cite{Blank,DiamondKloeden}. This seemingly unpretentious idea has turned out to be a source of intractable problems, most notably those concerning periodicity \cite{Vivaldi,AkiyamaPetho,HannuschPetho}. Added to the intractability of such problems is the literature's lack of a collection of tractable spatially discretised systems which could serve as toy models.


It is therefore appropriate for the literature to begin building a collection of such systems. A modest starting point has been provided by Rozikov et al.\ \cite{RozikovSattarovUsmonov}, who studied the one-parameter family of discretised one-dimensional linear maps given by
\begin{equation}\label{eq:mapsRozikov}
f:\mathbb{R}\to\mathbb{R},\qquad f(x)=\left\lfloor\lambda x\right\rfloor,
\end{equation}
where $\lambda\in\mathbb{R}$, constructible as above by letting $X=\mathbb{R}$, $F(x)=\lambda x$, and $D(x)=\left\lfloor x\right\rfloor$. For every $\lambda\in\mathbb{R}$, the authors presented in their analysis an explicit description of both the set $$\Fix(f):=\{x\in\mathbb{R}:f(x)=x\}$$ of all fixed points of $f$, and the $\omega$-limit set $$\omega_f(x):=\left\{\ell\in\mathbb{R}:\text{there exist }n_1,n_2,\ldots\in\mathbb{N}_0\text{ with }n_1<n_2<\cdots\text{ such that }f^{n_k}(x)\xrightarrow{k\to\infty}\ell\right\}$$ under $f$ of every $x\in\mathbb{R}$. (In the above definition, as also throughout this paper, the superscript $n_k$ denotes $n_k$-fold self-composition.)

The aim of the present paper is to carry out the same analysis for a larger family of maps, namely, the two-parameter family of discretised one-dimensional affine maps given by
\begin{equation}\label{eq:maps}
f:\mathbb{R}\to\mathbb{R},\qquad f(x)=\left\lfloor\lambda x+\mu\right\rfloor,
\end{equation}
where $\lambda,\mu\in\mathbb{R}$, of which not only the maps studied in \cite{RozikovSattarovUsmonov} but also some of those discussed in \cite{EiseleHadeler} form subfamilies. Adopting the terminology already introduced by Long and Chen \cite{LongChen} for a two-dimensional version of $f$, we refer to $f$ as the \textit{quasi-affine map} induced by the affine map $$F:\mathbb{R}\to\mathbb{R},\qquad F(x)=\lambda x+\mu.$$

This paper is organised as follows. In the upcoming section \ref{sec:mainresults}, we describe our main results. These are summarised in three theorems: Theorem \ref{thm:FP} which describes the set $\Fix(f)$ of all fixed points of $f$, Theorem \ref{thm:2cycles} which describes the set
$$\Cyc(f):=\left\{\{x,f(x)\}:f(x)\neq x\text{ and }f^2(x)= x\right\}$$
of all 2-cycles of $f$, and Theorem \ref{thm:omegalimit} which describes the $\omega$-limit set $\omega_f(x)$ of every $x\in\mathbb{R}$ under $f$. We also comment on several consequences and corollaries of these theorems. The subsequent sections contain proofs of these theorems: Theorems \ref{thm:FP} and \ref{thm:2cycles} in section \ref{sec:FP}, and Theorem \ref{thm:omegalimit} in the final section \ref{sec:LP}.

\section{Main results and consequences}\label{sec:mainresults}

As previously mentioned, our main results consist of three theorems. The first two theorems deal with the map's periodic points. First, we provide a complete description of its fixed points, thereby generalising \cite[Lemma 1]{RozikovSattarovUsmonov}.\medskip

\begin{samepage}
\begin{theorem}\label{thm:FP}\
\begin{enumerate}[leftmargin=0.8cm]
\item[(i)] If $\lambda>1$, then
$$\Fix(f)=\begin{cases}
\left\{\left\lceil-\frac{\mu}{\lambda-1}\right\rceil,\left\lceil-\frac{\mu}{\lambda-1}\right\rceil+1,\ldots,\left\lceil-\frac{\mu-1}{\lambda-1}\right\rceil-1\right\},&\text{if }\left\lceil-\frac{\mu}{\lambda-1}\right\rceil\leqslant\left\lceil-\frac{\mu-1}{\lambda-1}\right\rceil-1;\\
\varnothing,&\text{if }\left\lceil-\frac{\mu}{\lambda-1}\right\rceil>\left\lceil-\frac{\mu-1}{\lambda-1}\right\rceil-1.
\end{cases}$$
\item[(ii)] If $\lambda=1$, then
$$\Fix(f)=\begin{cases}
\mathbb{Z},&\text{if }0\leqslant\mu<1;\\
\varnothing,&\text{if }\mu<0\text{ or }\mu\geqslant 1.
\end{cases}$$
\item[(iii)] If $\lambda<1$, then
$$\Fix(f)=\begin{cases}
\left\{\left\lfloor-\frac{\mu-1}{\lambda-1}\right\rfloor+1,\left\lfloor-\frac{\mu-1}{\lambda-1}\right\rfloor+2,\ldots,\left\lfloor-\frac{\mu}{\lambda-1}\right\rfloor\right\},&\text{if }\left\lfloor-\frac{\mu-1}{\lambda-1}\right\rfloor+1\leqslant\left\lfloor-\frac{\mu}{\lambda-1}\right\rfloor;\\
\varnothing,&\text{if }\left\lfloor-\frac{\mu-1}{\lambda-1}\right\rfloor+1>\left\lfloor-\frac{\mu}{\lambda-1}\right\rfloor.
\end{cases}$$
\end{enumerate}
\end{theorem}\medskip
\end{samepage}

\begin{figure}
\centering
\includegraphics[height=7cm]{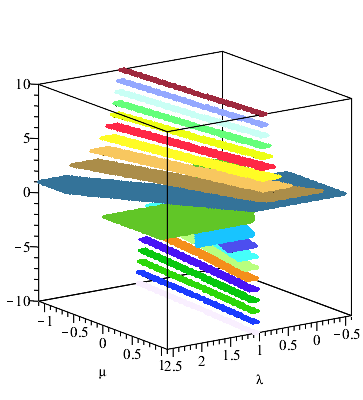}\qquad\qquad\includegraphics[height=6.5cm]{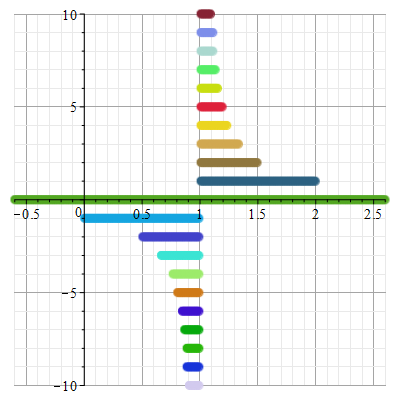}
\caption{\label{fig:bifurcation}The bifurcation diagrams of the fixed points of the map \eqref{eq:maps} (left) and of the map \eqref{eq:mapsRozikov} (right).}
\end{figure}

\noindent This theorem, proved in section \ref{sec:FP}, allows the generation of the codimension-two bifurcation diagram
$$\left\{(\lambda,\mu,x)\in\mathbb{R}^3:\left\lfloor\lambda x+\mu\right\rfloor=x\right\}$$
of the fixed points of $f$, shown in Figure \ref{fig:bifurcation} (left). Projecting this to the vertical plane $\mu=0$ gives the codimension-one bifurcation diagram
$$\left\{(\lambda,x)\in\mathbb{R}^2:\left\lfloor\lambda x\right\rfloor=x\right\}$$
of the fixed points of the map studied in \cite{RozikovSattarovUsmonov}, shown in Figure \ref{fig:bifurcation} (right). Notice the distinguished topological change occurring at $\lambda=1$. First, the trivial fixed point $0$ exists for all values of $\lambda$. As $\lambda\to 1^-$, we have $\left|\Fix(f)\cap\mathbb{Z}^-\right|\rightarrow\infty$ while $\Fix(f)\cap\mathbb{Z}^+=\varnothing$. On the other hand, as $\lambda\to 1^+$, we have $\left|\Fix(f)\cap\mathbb{Z}^+\right|\rightarrow\infty$ while $\Fix(f)\cap\mathbb{Z}^-=\varnothing$. At $\lambda=1$, we have $\Fix(f)=\mathbb{Z}$. To the best of the author's knowledge, no terminology has been given to such a bifurcation of fixed points, and indeed to any bifurcation undergone by spatially discretised systems, which therefore deserves further attention.
 

\begin{figure}
\centering
\includegraphics[height=7cm]{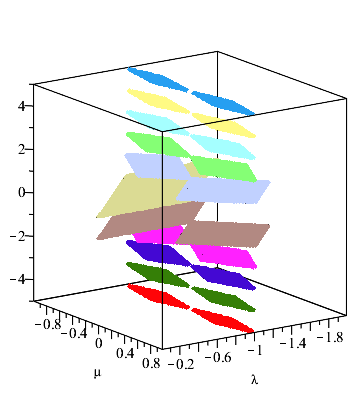}
\caption{\label{fig:bifurcationplane}The bifurcation diagram of the period-2 points of the map \eqref{eq:maps}.}
\end{figure}

Theorem \ref{thm:FP} also allows the formulation of an explicit formula for the number $N_p$ of fixed points of $f$ as a function of two variables $\lambda,\mu\in\mathbb{R}$:
$$N_p(\lambda,\mu)=\begin{cases}
\max\left\{0,\left\lceil-\frac{\mu-1}{\lambda-1}\right\rceil-\left\lceil-\frac{\mu}{\lambda-1}\right\rceil\right\},&\text{if }\lambda>1;\\
0,&\text{if }\lambda=1\text{ and }\mu\in(-\infty,0)\cup[1,\infty);\\
\infty,&\text{if }\lambda=1\text{ and }\mu\in[0,1);\\
\max\left\{0,\left\lfloor-\frac{\mu}{\lambda-1}\right\rfloor-\left\lfloor-\frac{\mu-1}{\lambda-1}\right\rfloor\right\},&\text{if }\lambda<1.
\end{cases}$$
Using the basic inequalities
$$x-1<\left\lfloor x\right\rfloor\leqslant x\qquad\text{and}\qquad x\leqslant\left\lceil x\right\rceil<x+1$$
valid for every $x\in\mathbb{R}$, one verifies that for $\lambda>1$ and for $\lambda<1$ we have, respectively,
$$N_p(\lambda,\mu)\geqslant\frac{1}{\lambda-1}-1\xrightarrow{\lambda\to1^+}\infty\qquad\text{and}\qquad N_p(\lambda,\mu)\geqslant-1-\frac{1}{\lambda-1}\xrightarrow{\lambda\to1^-}\infty,$$
for every $\mu\in\mathbb{R}$. Moreover, for every $\mu\in\mathbb{R}$, there exist values of $\lambda\in\mathbb{R}$ such that $f$ has exactly $N_p(\lambda,\mu)=n\in\mathbb{N}$ fixed points: among others, $\frac{n-1}{n}$ and $\frac{n+1}{n}$, as easily verified (cf.\ \cite[Lemma 1]{RozikovSattarovUsmonov}).

Our second main result gives an explicit description of the set $\Cyc(f)$ of all 2-cycles of $f$.\medskip

\begin{theorem}\label{thm:2cycles}
The set of all $2$-cycles of $f$ is given by
$$\Cyc(f)=\begin{cases}
\bigcup\limits_{k=1}^{\left\lceil\frac{1}{\lambda+1}\right\rceil-1}\left\{\left\{x,x+k\right\}:x\in\left\{\left\lfloor\frac{-\lambda k-\mu+1}{\lambda-1}\right\rfloor+1,\ldots,\left\lfloor\frac{k-\mu}{\lambda-1}\right\rfloor\right\}\right\},&\text{if }{-1<\lambda<0};\\
\left\{\left\{x,-x+\left\lfloor\mu\right\rfloor\right\}:x\in\mathbb{Z}\right\},&\text{if }\lambda=-1;\\
\bigcup\limits_{k=1}^{-\left\lfloor\frac{1}{\lambda+1}\right\rfloor-1}\left\{\left\{x,x+k\right\}:x\in\left\{\left\lfloor\frac{k+1-\mu}{\lambda-1}\right\rfloor+1,\ldots,\left\lfloor\frac{-\lambda k-\mu}{\lambda-1}\right\rfloor\right\}\right\},&\text{if }{-2<\lambda<-1};\\
\varnothing,&\text{if }\lambda\leqslant-2\text{ or }\lambda\geqslant 0.
\end{cases}$$
\end{theorem}\medskip

\noindent Notice that in the case $\lambda=-1$, every integer is a period-$2$ point. Theorem \ref{thm:2cycles}, also proved in section \ref{sec:FP}, implies that the codimension-two bifurcation diagram
$$\left\{(\lambda,\mu,x)\in\mathbb{R}^3:\left\lfloor\lambda x+\mu\right\rfloor\neq x\text{ and }\left\lfloor\lambda \left\lfloor\lambda x+\mu\right\rfloor+\mu\right\rfloor=x\right\}$$
of the period-2 points of $f$ is as displayed in Figure \ref{fig:bifurcationplane}, and that the number $N_c$ of $2$-cycles of $f$ as a function of $\lambda,\mu\in\mathbb{R}$ is given by
$$N_c(\lambda,\mu)=\begin{cases}
\sum\limits_{k=1}^{\left\lceil\frac{1}{\lambda+1}\right\rceil-1}\max\left\{0,\left\lfloor\frac{k-\mu}{\lambda-1}\right\rfloor-\left\lfloor\frac{-\lambda k-\mu+1}{\lambda-1}\right\rfloor\right\},&\text{if }{-1<\lambda<0};\\
\infty,&\text{if }\lambda=-1;\\
\sum\limits_{k=1}^{-\left\lfloor\frac{1}{\lambda+1}\right\rfloor-1}\max\left\{0,\left\lfloor\frac{-\lambda k-\mu}{\lambda-1}\right\rfloor-\left\lfloor\frac{k+1-\mu}{\lambda-1}\right\rfloor\right\},&\text{if }{-2<\lambda<-1};\\
0,&\text{if }\lambda\leqslant-2\text{ or }\lambda\geqslant 0.
\end{cases}$$

As the monotonicity of $f$ implies the non-existence of $n$-cycles for every $n\geqslant 3$, we next turn our attention to $\omega$-limit sets under the map. Our final theorem characterises $\omega_f(x)$ for every $x\in\mathbb{R}$, for various pairs of parameter values, thereby generalising \cite[Theorems 2--4]{RozikovSattarovUsmonov}.\medskip

\begin{theorem}\label{thm:omegalimit}\
\begin{enumerate}[leftmargin=1cm]
\item[(i)] If $\lambda>1$ and $\left\lceil-\frac{\mu}{\lambda-1}\right\rceil\leqslant\left\lceil-\frac{\mu-1}{\lambda-1}\right\rceil-1$, then for every $x\in\mathbb{R}$ we have
$$\omega_f(x)=\begin{cases}
\{\infty\},&\text{if }x\geqslant\frac{1}{\lambda}\left(\left\lceil-\frac{\mu-1}{\lambda-1}\right\rceil-\mu\right);\\
\{-\infty\},&\text{if }x<\frac{1}{\lambda}\left(\left\lceil-\frac{\mu}{\lambda-1}\right\rceil-\mu\right);\\
\{f(x)\},&\text{if }\frac{1}{\lambda}\left(\left\lceil-\frac{\mu}{\lambda-1}\right\rceil-\mu\right)\leqslant x<\frac{1}{\lambda}\left(\left\lceil-\frac{\mu-1}{\lambda-1}\right\rceil-\mu\right).
\end{cases}$$
\item[(ii)] If $\lambda>1$ and $\left\lceil-\frac{\mu}{\lambda-1}\right\rceil>\left\lceil-\frac{\mu-1}{\lambda-1}\right\rceil-1$, then for every $x\in\mathbb{R}$ we have
$$\omega_f(x)=\begin{cases}
\{\infty\},&\text{if }x\geqslant\frac{1}{\lambda}\left(\left\lfloor-\frac{\mu}{\lambda-1}\right\rfloor-\mu+1\right);\\
\{-\infty\},&\text{if }x<\frac{1}{\lambda}\left(\left\lfloor-\frac{\mu}{\lambda-1}\right\rfloor-\mu+1\right).
\end{cases}$$
\item[(iii)] If $\lambda=1$, then the following holds.
\begin{itemize}[leftmargin=0.5cm]
\item If $\mu\geqslant 1$, then for every $x\in\mathbb{R}$ we have $\omega_f(x)=\{\infty\}$.
\item If $\mu<0$, then for every $x\in\mathbb{R}$ we have $\omega_f(x)=\{-\infty\}$.
\item If $\mu\in[0,1)$, then for every $x\in\mathbb{R}$ we have $\omega_f(x)=\left\{\left\lfloor x+\mu\right\rfloor\right\}$.
\end{itemize}
\item[(iv)] If $0<\lambda<1$, then for every $x\in\mathbb{R}$ we have
$$\omega_f(x)=\begin{cases}
\left\{\left\lfloor-\frac{\mu-1}{\lambda-1}\right\rfloor+1\right\},&\text{if }x\leqslant-\frac{\mu-1}{\lambda-1};\\
\left\{\left\lfloor-\frac{\mu}{\lambda-1}\right\rfloor\right\},&\text{if }x>-\frac{\mu}{\lambda-1};\\
\{f(x)\},&\text{if }{-\frac{\mu-1}{\lambda-1}<x\leqslant-\frac{\mu}{\lambda-1}}.
\end{cases}$$
\item[(v)] If $\lambda=0$, then for every $x\in\mathbb{R}$ we have
$\omega_f(x)=\left\{\left\lfloor \mu\right\rfloor\right\}$.
\item[(vi)] If $-1<\lambda<0$ and $\left\lfloor-\frac{\mu-1}{\lambda-1}\right\rfloor+1\leqslant\left\lfloor-\frac{\mu}{\lambda-1}\right\rfloor$, then the following holds.
\begin{itemize}[leftmargin=0.5cm]
\item If $\frac{1}{\lambda}\left(\left\lfloor-\frac{\mu}{\lambda-1}\right\rfloor-\mu+1\right)<x\leqslant\frac{1}{\lambda}\left(\left\lfloor-\frac{\mu}{\lambda-1}\right\rfloor-\mu\right)$, then $\omega_f(x)=\left\{\left\lfloor-\frac{\mu}{\lambda-1}\right\rfloor\right\}$.
\item If $x\leqslant\frac{1}{\lambda}\left(\left\lfloor-\frac{\mu}{\lambda-1}\right\rfloor-\mu+1\right)$ or $x>\frac{1}{\lambda}\left(\left\lfloor-\frac{\mu}{\lambda-1}\right\rfloor-\mu\right)$, then $\omega_f(x)$ is either a $2$-cycle or $\left\{\left\lfloor-\frac{\mu}{\lambda-1}\right\rfloor\right\}$.
\end{itemize}
\item[(vii)] If $-1<\lambda<0$ and $\left\lfloor-\frac{\mu-1}{\lambda-1}\right\rfloor+1>\left\lfloor-\frac{\mu}{\lambda-1}\right\rfloor$, then for every $x\in\mathbb{R}$, $\omega_f(x)$ is a $2$-cycle.
\item[(viii)] If $\lambda\leqslant -1$ and $\left\lfloor-\frac{\mu-1}{\lambda-1}\right\rfloor+1\leqslant\left\lfloor-\frac{\mu}{\lambda-1}\right\rfloor$, then the following holds.
\begin{itemize}[leftmargin=0.5cm]
\item If $\frac{1}{\lambda}\left(\left\lfloor-\frac{\mu}{\lambda-1}\right\rfloor-\mu+1\right)<x\leqslant \frac{1}{\lambda}\left(\left\lfloor-\frac{\mu}{\lambda-1}\right\rfloor-\mu\right)$, then $\omega_f(x)=\left\{\left\lfloor-\frac{\mu}{\lambda-1}\right\rfloor\right\}$.
\item If $x\leqslant\frac{1}{\lambda}\left(\left\lfloor-\frac{\mu}{\lambda-1}\right\rfloor-\mu+1\right)$ or $x>\frac{1}{\lambda}\left(\left\lfloor-\frac{\mu}{\lambda-1}\right\rfloor-\mu\right)$, then $\omega_f(x)$ is either a $2$-cycle or $\{-\infty,\infty\}$.
\end{itemize}
\item[(ix)] If $\lambda\leqslant-1$ and $\left\lfloor-\frac{\mu-1}{\lambda-1}\right\rfloor+1>\left\lfloor-\frac{\mu}{\lambda-1}\right\rfloor$, then for every $x\in\mathbb{R}$, $\omega_f(x)$ is either a $2$-cycle or $\{-\infty,\infty\}$.
\end{enumerate}
\end{theorem}\medskip

By Theorem \ref{thm:2cycles}, if $\lambda\leqslant -2$, then no 2-cycles exist, and so we have the following special case of parts (viii) and (ix).\medskip

\begin{corollary}
Suppose $\lambda\leqslant -2$. If $\left\lfloor-\frac{\mu-1}{\lambda-1}\right\rfloor+1\leqslant\left\lfloor-\frac{\mu}{\lambda-1}\right\rfloor$, then the following holds.
\begin{itemize}[leftmargin=0.5cm]
\item If $\frac{1}{\lambda}\left(\left\lfloor-\frac{\mu}{\lambda-1}\right\rfloor-\mu+1\right)<x\leqslant \frac{1}{\lambda}\left(\left\lfloor-\frac{\mu}{\lambda-1}\right\rfloor-\mu\right)$, then $\omega_f(x)=\left\{\left\lfloor-\frac{\mu}{\lambda-1}\right\rfloor\right\}$.
\item If $x\leqslant\frac{1}{\lambda}\left(\left\lfloor-\frac{\mu}{\lambda-1}\right\rfloor-\mu+1\right)$ or $x>\frac{1}{\lambda}\left(\left\lfloor-\frac{\mu}{\lambda-1}\right\rfloor-\mu\right)$, then $\omega_f(x)=\{-\infty,\infty\}$.
\end{itemize}
If $\left\lfloor-\frac{\mu-1}{\lambda-1}\right\rfloor+1>\left\lfloor-\frac{\mu}{\lambda-1}\right\rfloor$, then for every $x\in\mathbb{R}$ we have $\omega_f(x)=\{-\infty,\infty\}$.
\end{corollary}\medskip

\noindent In addition, if $\lambda=-1$, then direct computation shows that for every $x\in \mathbb{R}$ we have $f(x)=\left\lfloor-x+\mu\right\rfloor$, $f^2(x)=-\left\lfloor-x+\mu\right\rfloor+\left\lfloor\mu\right\rfloor$, and $f^3(x)=f(x)$, giving rise to the following corollary.\medskip

\begin{corollary}
Suppose $\lambda=-1$. For every $x\in\mathbb{R}$ we have $$\omega_f(x)=\left\{\left\lfloor-x+\mu\right\rfloor,-\left\lfloor-x+\mu\right\rfloor+\left\lfloor\mu\right\rfloor\right\}.$$
\end{corollary}\medskip

Cobweb diagrams of $f$ which visualise the map's orbital behaviour in each of the above cases are presented in Figure \ref{fig:cobweb}. In the final section \ref{sec:LP}, we shall prove Theorem \ref{thm:omegalimit} by dividing it into a number of propositions.


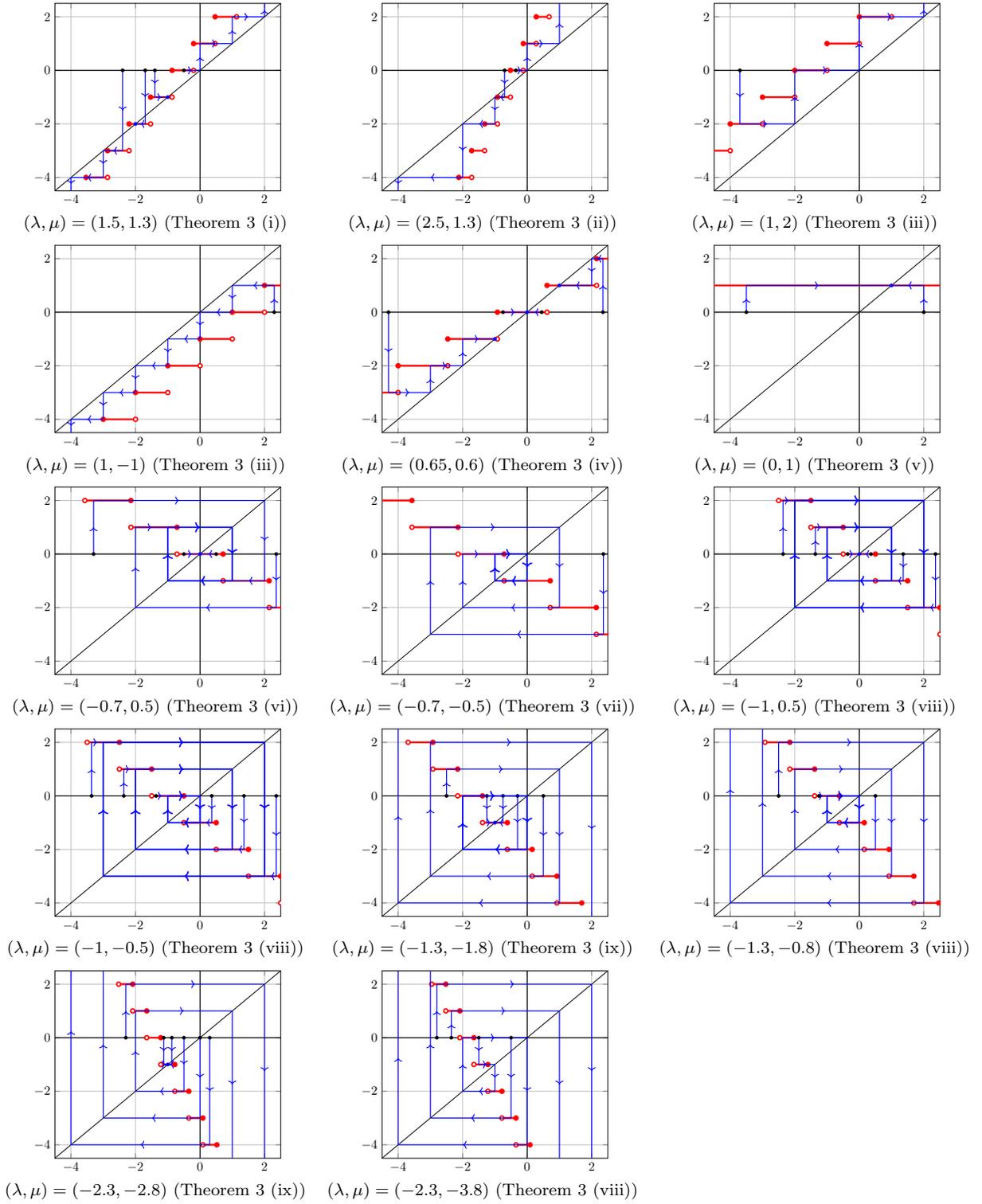
\begin{figure}\centering
\begin{tabular}{ccc}
		\begin{tikzpicture} [scale=0.55] 
\begin{axis}[
	xmin=-4.5,
	xmax=2.5,
	ymin=-4.5,
	ymax=2.5,
    xtick={-4,-2,0,2},
    ytick={-4,-2,0,2},
	xmajorgrids,
	ymajorgrids
]

\draw[thick] (axis cs:-4.5,0) -- (axis cs:2.5,0);
\draw[thick] (axis cs:0,-4.5) -- (axis cs:0,2.5);

\addplot[domain=-4.5:2.5] (\x,\x);

\addplot[red,ultra thick,domain=-3.533333333:-2.866666667] (\x,-4);
\draw[red,very thick,fill=red] (axis cs:-3.533333333,-4) circle (1.75pt);
\draw[red,very thick,fill=white] (axis cs:-2.866666667,-4) circle (1.75pt);

\addplot[red,ultra thick,domain=-2.866666667:-2.200000000] (\x,-3);
\draw[red,very thick,fill=red] (axis cs:-2.866666667,-3) circle (1.75pt);
\draw[red,very thick,fill=white] (axis cs:-2.200000000,-3) circle (1.75pt);

\addplot[red,ultra thick,domain=-2.200000000:-1.533333333] (\x,-2);
\draw[red,very thick,fill=red] (axis cs:-2.200000000,-2) circle (1.75pt);
\draw[red,very thick,fill=white] (axis cs:-1.533333333,-2) circle (1.75pt);

\addplot[red,ultra thick,domain=-1.533333333:-.8666666667] (\x,-1);
\draw[red,very thick,fill=red] (axis cs:-1.533333333,-1) circle (1.75pt);
\draw[red,very thick,fill=white] (axis cs:-.8666666667,-1) circle (1.75pt);

\addplot[red,ultra thick,domain=-.8666666667:-.2000000000] (\x,0);
\draw[red,very thick,fill=red] (axis cs:-.8666666667,0) circle (1.75pt);
\draw[red,very thick,fill=white] (axis cs:-.2000000000,0) circle (1.75pt);

\addplot[red,ultra thick,domain=-.2000000000:.4666666667] (\x,1);
\draw[red,very thick,fill=red] (axis cs:-.2000000000,1) circle (1.75pt);
\draw[red,very thick,fill=white] (axis cs:.4666666667,1) circle (1.75pt);

\addplot[red,ultra thick,domain=.4666666667:1.133333333] (\x,2);
\draw[red,very thick,fill=red] (axis cs:.4666666667,2) circle (1.75pt);
\draw[red,very thick,fill=white] (axis cs:1.133333333,2) circle (1.75pt);

\begin{scope}[decoration={
    markings,
    mark=at position 0.5 with {\arrow{>}}}
    ] 
\draw[thick,blue,postaction={decorate}] (axis cs:-0.5,0) -- (axis cs:0,0);
\draw[thick,blue,postaction={decorate}] (axis cs:0,0) -- (axis cs:0,1);
\draw[thick,blue,postaction={decorate}] (axis cs:0,1) -- (axis cs:1,1);
\draw[thick,blue,postaction={decorate}] (axis cs:1,1) -- (axis cs:1,2);
\draw[thick,blue,postaction={decorate}] (axis cs:1,2) -- (axis cs:2,2);
\draw[thick,blue,postaction={decorate}] (axis cs:2,2) -- (axis cs:2,2.5);

\draw[thick,blue,postaction={decorate}] (axis cs:-1.4,0) -- (axis cs:-1.4,-1);
\draw[thick,blue,postaction={decorate}] (axis cs:-1.4,-1) -- (axis cs:-1,-1);

\draw[thick,blue,postaction={decorate}] (axis cs:-1.7,0) -- (axis cs:-1.7,-2);
\draw[thick,blue,postaction={decorate}] (axis cs:-1.7,-2) -- (axis cs:-2,-2);

\draw[thick,blue,postaction={decorate}] (axis cs:-2.4,0) -- (axis cs:-2.4,-3);
\draw[thick,blue,postaction={decorate}] (axis cs:-2.4,-3) -- (axis cs:-3,-3);
\draw[thick,blue,postaction={decorate}] (axis cs:-3,-3) -- (axis cs:-3,-4);
\draw[thick,blue,postaction={decorate}] (axis cs:-3,-4) -- (axis cs:-4,-4);
\draw[thick,blue,postaction={decorate}] (axis cs:-4,-4) -- (axis cs:-4,-4.5);

\fill[black] (axis cs:-0.5,0) circle (1.625pt);
\fill[black] (axis cs:-1.4,0) circle (1.625pt);
\fill[black] (axis cs:-1.7,0) circle (1.625pt);
\fill[black] (axis cs:-2.4,0) circle (1.625pt);

\fill[blue] (axis cs:-1,-1) circle (1.625pt);
\fill[blue] (axis cs:-2,-2) circle (1.625pt);
\end{scope}
\end{axis}
\end{tikzpicture}&
		\begin{tikzpicture} [scale=0.55] 
\begin{axis}[
	xmin=-4.5,
	xmax=2.5,
	ymin=-4.5,
	ymax=2.5,
    xtick={-4,-2,0,2},
    ytick={-4,-2,0,2},
	xmajorgrids,
	ymajorgrids
]

\draw[thick] (axis cs:-4.5,0) -- (axis cs:2.5,0);
\draw[thick] (axis cs:0,-4.5) -- (axis cs:0,2.5);

\addplot[domain=-4.5:2.5] (\x,\x);

\addplot[red,ultra thick,domain=-2.120000000:-1.720000000] (\x,-4);
\draw[red,very thick,fill=red] (axis cs:-2.120000000,-4) circle (1.75pt);
\draw[red,very thick,fill=white] (axis cs:-1.720000000,-4) circle (1.75pt);

\addplot[red,ultra thick,domain=-1.720000000:-1.320000000] (\x,-3);
\draw[red,very thick,fill=red] (axis cs:-1.720000000,-3) circle (1.75pt);
\draw[red,very thick,fill=white] (axis cs:-1.320000000,-3) circle (1.75pt);

\addplot[red,ultra thick,domain=-1.320000000:-.9200000000] (\x,-2);
\draw[red,very thick,fill=red] (axis cs:-1.320000000,-2) circle (1.75pt);
\draw[red,very thick,fill=white] (axis cs:-.9200000000,-2) circle (1.75pt);

\addplot[red,ultra thick,domain=-.9200000000:-.5200000000] (\x,-1);
\draw[red,very thick,fill=red] (axis cs:-.9200000000,-1) circle (1.75pt);
\draw[red,very thick,fill=white] (axis cs:-.5200000000,-1) circle (1.75pt);

\addplot[red,ultra thick,domain=-.5200000000:-.1200000000] (\x,0);
\draw[red,very thick,fill=red] (axis cs:-.5200000000,0) circle (1.75pt);
\draw[red,very thick,fill=white] (axis cs:-.1200000000,0) circle (1.75pt);

\addplot[red,ultra thick,domain=-.1200000000:.2800000000] (\x,1);
\draw[red,very thick,fill=red] (axis cs:-.1200000000,1) circle (1.75pt);
\draw[red,very thick,fill=white] (axis cs:.2800000000,1) circle (1.75pt);

\addplot[red,ultra thick,domain=.2800000000:.6800000000] (\x,2);
\draw[red,very thick,fill=red] (axis cs:.2800000000,2) circle (1.75pt);
\draw[red,very thick,fill=white] (axis cs:.6800000000,2) circle (1.75pt);

\begin{scope}[decoration={
    markings,
    mark=at position 0.5 with {\arrow{>}}}
    ] 
\draw[thick,blue,postaction={decorate}] (axis cs:-0.35,0) -- (axis cs:0,0);
\draw[thick,blue,postaction={decorate}] (axis cs:0,0) -- (axis cs:0,1);
\draw[thick,blue,postaction={decorate}] (axis cs:0,1) -- (axis cs:1,1);
\draw[thick,blue,postaction={decorate}] (axis cs:1,1) -- (axis cs:1,2.5);

\draw[thick,blue,postaction={decorate}] (axis cs:-0.7,0) -- (axis cs:-0.7,-1);
\draw[thick,blue,postaction={decorate}] (axis cs:-0.7,-1) -- (axis cs:-1,-1);
\draw[thick,blue,postaction={decorate}] (axis cs:-1,-1) -- (axis cs:-1,-2);
\draw[thick,blue,postaction={decorate}] (axis cs:-1,-2) -- (axis cs:-2,-2);
\draw[thick,blue,postaction={decorate}] (axis cs:-2,-2) -- (axis cs:-2,-4);
\draw[thick,blue,postaction={decorate}] (axis cs:-2,-4) -- (axis cs:-4,-4);
\draw[thick,blue,postaction={decorate}] (axis cs:-4,-4) -- (axis cs:-4,-4.5);

\fill[black] (axis cs:-0.35,0) circle (1.625pt);
\fill[black] (axis cs:-0.7,0) circle (1.625pt);
\end{scope}
\end{axis}
\end{tikzpicture}&
		\begin{tikzpicture} [scale=0.55] 
\begin{axis}[
	xmin=-4.5,
	xmax=2.5,
	ymin=-4.5,
	ymax=2.5,
    xtick={-4,-2,0,2},
    ytick={-4,-2,0,2},
	xmajorgrids,
	ymajorgrids
]

\draw[thick] (axis cs:-4.5,0) -- (axis cs:2.5,0);
\draw[thick] (axis cs:0,-4.5) -- (axis cs:0,2.5);

\addplot[domain=-4.5:2.5] (\x,\x);

\addplot[red,ultra thick,domain=-6:-5] (\x,-4);
\draw[red,very thick,fill=red] (axis cs:-6,-4) circle (1.75pt);
\draw[red,very thick,fill=white] (axis cs:-5,-4) circle (1.75pt);

\addplot[red,ultra thick,domain=-5:-4] (\x,-3);
\draw[red,very thick,fill=red] (axis cs:-5,-3) circle (1.75pt);
\draw[red,very thick,fill=white] (axis cs:-4,-3) circle (1.75pt);

\addplot[red,ultra thick,domain=-4:-3] (\x,-2);
\draw[red,very thick,fill=red] (axis cs:-4,-2) circle (1.75pt);
\draw[red,very thick,fill=white] (axis cs:-3,-2) circle (1.75pt);

\addplot[red,ultra thick,domain=-3:-2] (\x,-1);
\draw[red,very thick,fill=red] (axis cs:-3,-1) circle (1.75pt);
\draw[red,very thick,fill=white] (axis cs:-2,-1) circle (1.75pt);

\addplot[red,ultra thick,domain=-2:-1] (\x,0);
\draw[red,very thick,fill=red] (axis cs:-2,0) circle (1.75pt);
\draw[red,very thick,fill=white] (axis cs:-1,0) circle (1.75pt);

\addplot[red,ultra thick,domain=-1:0] (\x,1);
\draw[red,very thick,fill=red] (axis cs:-1,1) circle (1.75pt);
\draw[red,very thick,fill=white] (axis cs:0,1) circle (1.75pt);

\addplot[red,ultra thick,domain=0:1] (\x,2);
\draw[red,very thick,fill=red] (axis cs:0,2) circle (1.75pt);
\draw[red,very thick,fill=white] (axis cs:1,2) circle (1.75pt);

\begin{scope}[decoration={
    markings,
    mark=at position 0.5 with {\arrow{>}}}
    ] 
\draw[thick,blue,postaction={decorate}] (axis cs:-3.7,0) -- (axis cs:-3.7,-2);
\draw[thick,blue,postaction={decorate}] (axis cs:-3.7,-2) -- (axis cs:-2,-2);
\draw[thick,blue,postaction={decorate}] (axis cs:-2,-2) -- (axis cs:-2,0);
\draw[thick,blue,postaction={decorate}] (axis cs:-2,0) -- (axis cs:0,0);
\draw[thick,blue,postaction={decorate}] (axis cs:0,0) -- (axis cs:0,2);
\draw[thick,blue,postaction={decorate}] (axis cs:0,2) -- (axis cs:2,2);
\draw[thick,blue,postaction={decorate}] (axis cs:2,2) -- (axis cs:2,2.5);

\fill[black] (axis cs:-3.7,0) circle (1.625pt);
\end{scope}
\end{axis}
\end{tikzpicture}\\[-0.1375cm]
		{\scriptsize $(\lambda,\mu)=(1.5,1.3)$ (Theorem \ref{thm:omegalimit} (i))} & {\scriptsize $(\lambda,\mu)=(2.5,1.3)$ (Theorem \ref{thm:omegalimit} (ii))} & {\scriptsize $(\lambda,\mu)=(1,2)$ (Theorem \ref{thm:omegalimit} (iii))}\\[0.15cm]
		
		\begin{tikzpicture} [scale=0.55] 
\begin{axis}[
	xmin=-4.5,
	xmax=2.5,
	ymin=-4.5,
	ymax=2.5,
    xtick={-4,-2,0,2},
    ytick={-4,-2,0,2},
	xmajorgrids,
	ymajorgrids
]

\draw[thick] (axis cs:-4.5,0) -- (axis cs:2.5,0);
\draw[thick] (axis cs:0,-4.5) -- (axis cs:0,2.5);

\addplot[domain=-4.5:2.5] (\x,\x);

\addplot[red,ultra thick,domain=-3:-2] (\x,-4);
\draw[red,very thick,fill=red] (axis cs:-3,-4) circle (1.75pt);
\draw[red,very thick,fill=white] (axis cs:-2,-4) circle (1.75pt);

\addplot[red,ultra thick,domain=-2:-1] (\x,-3);
\draw[red,very thick,fill=red] (axis cs:-2,-3) circle (1.75pt);
\draw[red,very thick,fill=white] (axis cs:-1,-3) circle (1.75pt);

\addplot[red,ultra thick,domain=-1:0] (\x,-2);
\draw[red,very thick,fill=red] (axis cs:-1,-2) circle (1.75pt);
\draw[red,very thick,fill=white] (axis cs:0,-2) circle (1.75pt);

\addplot[red,ultra thick,domain=0:1] (\x,-1);
\draw[red,very thick,fill=red] (axis cs:0,-1) circle (1.75pt);
\draw[red,very thick,fill=white] (axis cs:1,-1) circle (1.75pt);

\addplot[red,ultra thick,domain=1:2] (\x,0);
\draw[red,very thick,fill=red] (axis cs:1,0) circle (1.75pt);
\draw[red,very thick,fill=white] (axis cs:2,0) circle (1.75pt);

\addplot[red,ultra thick,domain=2:3] (\x,1);
\draw[red,very thick,fill=red] (axis cs:2,1) circle (1.75pt);
\draw[red,very thick,fill=white] (axis cs:3,1) circle (1.75pt);

\begin{scope}[decoration={
    markings,
    mark=at position 0.5 with {\arrow{>}}}
    ] 
\draw[thick,blue,postaction={decorate}] (axis cs:2.3,0) -- (axis cs:2.3,1);
\draw[thick,blue,postaction={decorate}] (axis cs:2.3,1) -- (axis cs:1,1);
\draw[thick,blue,postaction={decorate}] (axis cs:1,1) -- (axis cs:1,0);
\draw[thick,blue,postaction={decorate}] (axis cs:1,0) -- (axis cs:0,0);
\draw[thick,blue,postaction={decorate}] (axis cs:0,0) -- (axis cs:0,-1);
\draw[thick,blue,postaction={decorate}] (axis cs:0,-1) -- (axis cs:-1,-1);
\draw[thick,blue,postaction={decorate}] (axis cs:-1,-1) -- (axis cs:-1,-2);
\draw[thick,blue,postaction={decorate}] (axis cs:-1,-2) -- (axis cs:-2,-2);
\draw[thick,blue,postaction={decorate}] (axis cs:-2,-2) -- (axis cs:-2,-3);
\draw[thick,blue,postaction={decorate}] (axis cs:-2,-3) -- (axis cs:-3,-3);
\draw[thick,blue,postaction={decorate}] (axis cs:-3,-3) -- (axis cs:-3,-4);
\draw[thick,blue,postaction={decorate}] (axis cs:-3,-4) -- (axis cs:-4,-4);
\draw[thick,blue,postaction={decorate}] (axis cs:-4,-4) -- (axis cs:-4,-4.5);

\fill[black] (axis cs:2.3,0) circle (1.625pt);
\end{scope}
\end{axis}
\end{tikzpicture}&
		\begin{tikzpicture} [scale=0.55] 
\begin{axis}[
	xmin=-4.5,
	xmax=2.5,
	ymin=-4.5,
	ymax=2.5,
    xtick={-4,-2,0,2},
    ytick={-4,-2,0,2},
	xmajorgrids,
	ymajorgrids
]

\draw[thick] (axis cs:-4.5,0) -- (axis cs:2.5,0);
\draw[thick] (axis cs:0,-4.5) -- (axis cs:0,2.5);

\addplot[domain=-4.5:2.5] (\x,\x);

\addplot[red,ultra thick,domain=-5.538461538:-4] (\x,-3);
\draw[red,very thick,fill=red] (axis cs:-5.538461538,-3) circle (1.75pt);
\draw[red,very thick,fill=white] (axis cs:-4,-3) circle (1.75pt);

\addplot[red,ultra thick,domain=-4:-2.461538462] (\x,-2);
\draw[red,very thick,fill=red] (axis cs:-4,-2) circle (1.75pt);
\draw[red,very thick,fill=white] (axis cs:-2.461538462,-2) circle (1.75pt);

\addplot[red,ultra thick,domain=-2.461538462:-.9230769231] (\x,-1);
\draw[red,very thick,fill=red] (axis cs:-2.461538462,-1) circle (1.75pt);
\draw[red,very thick,fill=white] (axis cs:-.9230769231,-1) circle (1.75pt);

\addplot[red,ultra thick,domain=-.9230769231:.6153846154] (\x,0);
\draw[red,very thick,fill=red] (axis cs:-.9230769231,0) circle (1.75pt);
\draw[red,very thick,fill=white] (axis cs:.6153846154,0) circle (1.75pt);

\addplot[red,ultra thick,domain=.6153846154:2.153846154] (\x,1);
\draw[red,very thick,fill=red] (axis cs:.6153846154,1) circle (1.75pt);
\draw[red,very thick,fill=white] (axis cs:2.153846154,1) circle (1.75pt);

\addplot[red,ultra thick,domain=2.153846154:3.692307692] (\x,2);
\draw[red,very thick,fill=red] (axis cs:2.153846154,2) circle (1.75pt);
\draw[red,very thick,fill=white] (axis cs:3.692307692,2) circle (1.75pt);

\begin{scope}[decoration={
    markings,
    mark=at position 0.5 with {\arrow{>}}}
    ] 
\draw[thick,blue,postaction={decorate}] (axis cs:2.35,0) -- (axis cs:2.35,2);
\draw[thick,blue,postaction={decorate}] (axis cs:2.35,2) -- (axis cs:2,2);
\draw[thick,blue,postaction={decorate}] (axis cs:2,2) -- (axis cs:2,1);
\draw[thick,blue,postaction={decorate}] (axis cs:2,1) -- (axis cs:1,1);

\draw[thick,blue,postaction={decorate}] (axis cs:-0.75,0) -- (axis cs:0,0);
\draw[thick,blue,postaction={decorate}] (axis cs:0.45,0) -- (axis cs:0,0);

\draw[thick,blue,postaction={decorate}] (axis cs:-4.3,0) -- (axis cs:-4.3,-3);
\draw[thick,blue,postaction={decorate}] (axis cs:-4.3,-3) -- (axis cs:-3,-3);
\draw[thick,blue,postaction={decorate}] (axis cs:-3,-3) -- (axis cs:-3,-2);
\draw[thick,blue,postaction={decorate}] (axis cs:-3,-2) -- (axis cs:-2,-2);
\draw[thick,blue,postaction={decorate}] (axis cs:-2,-2) -- (axis cs:-2,-1);
\draw[thick,blue,postaction={decorate}] (axis cs:-2,-1) -- (axis cs:-1,-1);

\fill[black] (axis cs:-4.3,0) circle (1.625pt);
\fill[black] (axis cs:-0.75,0) circle (1.625pt);
\fill[black] (axis cs:0.45,0) circle (1.625pt);
\fill[black] (axis cs:2.35,0) circle (1.625pt);

\fill[blue] (axis cs:1,1) circle (1.625pt);
\fill[blue] (axis cs:0,0) circle (1.625pt);
\fill[blue] (axis cs:-1,-1) circle (1.625pt);
\end{scope}
\end{axis}
\end{tikzpicture}&
		\begin{tikzpicture} [scale=0.55] 
\begin{axis}[
	xmin=-4.5,
	xmax=2.5,
	ymin=-4.5,
	ymax=2.5,
    xtick={-4,-2,0,2},
    ytick={-4,-2,0,2},
	xmajorgrids,
	ymajorgrids
]

\draw[thick] (axis cs:-4.5,0) -- (axis cs:2.5,0);
\draw[thick] (axis cs:0,-4.5) -- (axis cs:0,2.5);

\addplot[domain=-4.5:2.5] (\x,\x);

\addplot[red,ultra thick,domain=-4.5:2.5] (\x,1);

\begin{scope}[decoration={
    markings,
    mark=at position 0.5 with {\arrow{>}}}
    ] 
\draw[thick,blue,postaction={decorate}] (axis cs:2,0) -- (axis cs:2,1);
\draw[thick,blue,postaction={decorate}] (axis cs:2,1) -- (axis cs:1,1);

\draw[thick,blue,postaction={decorate}] (axis cs:-3.5,0) -- (axis cs:-3.5,1);
\draw[thick,blue,postaction={decorate}] (axis cs:-3.5,1) -- (axis cs:1,1);

\fill[black] (axis cs:2,0) circle (1.625pt);
\fill[black] (axis cs:-3.5,0) circle (1.625pt);

\fill[blue] (axis cs:1,1) circle (1.625pt);
\end{scope}
\end{axis}
\end{tikzpicture}\\[-0.1375cm]
		{\scriptsize $(\lambda,\mu)=(1,-1)$ (Theorem \ref{thm:omegalimit} (iii))} & {\scriptsize $(\lambda,\mu)=(0.65,0.6)$ (Theorem \ref{thm:omegalimit} (iv))} & {\scriptsize $(\lambda,\mu)=(0,1)$ (Theorem \ref{thm:omegalimit} (v))}\\[0.15cm]

		\begin{tikzpicture} [scale=0.55] 
\begin{axis}[
	xmin=-4.5,
	xmax=2.5,
	ymin=-4.5,
	ymax=2.5,
    xtick={-4,-2,0,2},
    ytick={-4,-2,0,2},
	xmajorgrids,
	ymajorgrids
]

\draw[thick] (axis cs:-4.5,0) -- (axis cs:2.5,0);
\draw[thick] (axis cs:0,-4.5) -- (axis cs:0,2.5);

\addplot[domain=-4.5:2.5] (\x,\x);

\addplot[red,ultra thick,domain=2.142857143:3.571428571] (\x,-2);
\draw[red,very thick,fill=white] (axis cs:2.142857143,-2) circle (1.75pt);
\draw[red,very thick,fill=red] (axis cs:3.571428571,-2) circle (1.75pt);

\addplot[red,ultra thick,domain=.7142857143:2.142857143] (\x,-1);
\draw[red,very thick,fill=white] (axis cs:.7142857143,-1) circle (1.75pt);
\draw[red,very thick,fill=red] (axis cs:2.142857143,-1) circle (1.75pt);

\addplot[red,ultra thick,domain=-.7142857143:.7142857143] (\x,0);
\draw[red,very thick,fill=white] (axis cs:-.7142857143,0) circle (1.75pt);
\draw[red,very thick,fill=red] (axis cs:.7142857143,0) circle (1.75pt);

\addplot[red,ultra thick,domain=-2.142857143:-.7142857143] (\x,1);
\draw[red,very thick,fill=white] (axis cs:-2.142857143,1) circle (1.75pt);
\draw[red,very thick,fill=red] (axis cs:-.7142857143,1) circle (1.75pt);

\addplot[red,ultra thick,domain=-3.571428571:-2.142857143] (\x,2);
\draw[red,very thick,fill=white] (axis cs:-3.571428571,2) circle (1.75pt);
\draw[red,very thick,fill=red] (axis cs:-2.142857143,2) circle (1.75pt);

\begin{scope}[decoration={
    markings,
    mark=at position 0.5 with {\arrow{>}}}
    ] 
\draw[thick,blue,postaction={decorate}] (axis cs:0.5,0) -- (axis cs:0,0);

\draw[thick,blue,postaction={decorate}] (axis cs:-0.5,0) -- (axis cs:0,0);

\draw[thick,blue,postaction={decorate}] (axis cs:-3.3,0) -- (axis cs:-3.3,2);
\draw[thick,blue,postaction={decorate}] (axis cs:-3.3,2) -- (axis cs:2,2);
\draw[thick,blue,postaction={decorate}] (axis cs:2,2) -- (axis cs:2,-1);
\draw[thick,blue,postaction={decorate}] (axis cs:2,-1) -- (axis cs:1,-1);

\draw[thick,blue,postaction={decorate}] (axis cs:2.3625,0) -- (axis cs:2.3625,-2);
\draw[thick,blue,postaction={decorate}] (axis cs:2.3625,-2) -- (axis cs:-2,-2);
\draw[thick,blue,postaction={decorate}] (axis cs:-2,-2) -- (axis cs:-2,1);
\draw[thick,blue,postaction={decorate}] (axis cs:-2,1) -- (axis cs:-1,1);

\fill[black] (axis cs:0.5,0) circle (1.625pt);
\fill[black] (axis cs:-0.5,0) circle (1.625pt);
\fill[black] (axis cs:2.3625,0) circle (1.625pt);
\fill[black] (axis cs:-3.3,0) circle (1.625pt);

\fill[blue] (axis cs:0,0) circle (1.625pt);

\draw[very thick,blue,postaction={decorate}] (axis cs:1,-1) -- (axis cs:-1,-1);
\draw[very thick,blue,postaction={decorate}] (axis cs:-1,-1) -- (axis cs:-1,1);
\draw[very thick,blue,postaction={decorate}] (axis cs:-1,1) -- (axis cs:1,1);
\draw[very thick,blue,postaction={decorate}] (axis cs:1,1) -- (axis cs:1,-1);
\end{scope}
\end{axis}
\end{tikzpicture}&\begin{tikzpicture} [scale=0.55] 
\begin{axis}[
	xmin=-4.5,
	xmax=2.5,
	ymin=-4.5,
	ymax=2.5,
    xtick={-4,-2,0,2},
    ytick={-4,-2,0,2},
	xmajorgrids,
	ymajorgrids
]

\draw[thick] (axis cs:-4.5,0) -- (axis cs:2.5,0);
\draw[thick] (axis cs:0,-4.5) -- (axis cs:0,2.5);

\addplot[domain=-4.5:2.5] (\x,\x);

\addplot[red,ultra thick,domain=2.142857143:3.571428571] (\x,-3);
\draw[red,very thick,fill=white] (axis cs:2.142857143,-3) circle (1.75pt);
\draw[red,very thick,fill=red] (axis cs:3.571428571,-3) circle (1.75pt);

\addplot[red,ultra thick,domain=.7142857143:2.142857143] (\x,-2);
\draw[red,very thick,fill=white] (axis cs:.7142857143,-2) circle (1.75pt);
\draw[red,very thick,fill=red] (axis cs:2.142857143,-2) circle (1.75pt);

\addplot[red,ultra thick,domain=-.7142857143:.7142857143] (\x,-1);
\draw[red,very thick,fill=white] (axis cs:-.7142857143,-1) circle (1.75pt);
\draw[red,very thick,fill=red] (axis cs:.7142857143,-1) circle (1.75pt);

\addplot[red,ultra thick,domain=-2.142857143:-.7142857143] (\x,0);
\draw[red,very thick,fill=white] (axis cs:-2.142857143,0) circle (1.75pt);
\draw[red,very thick,fill=red] (axis cs:-.7142857143,0) circle (1.75pt);

\addplot[red,ultra thick,domain=-3.571428571:-2.142857143] (\x,1);
\draw[red,very thick,fill=white] (axis cs:-3.571428571,1) circle (1.75pt);
\draw[red,very thick,fill=red] (axis cs:-2.142857143,1) circle (1.75pt);

\addplot[red,ultra thick,domain=-5:-3.571428571] (\x,2);
\draw[red,very thick,fill=white] (axis cs:-5,2) circle (1.75pt);
\draw[red,very thick,fill=red] (axis cs:-3.571428571,2) circle (1.75pt);

\begin{scope}[decoration={
    markings,
    mark=at position 0.5 with {\arrow{>}}}
    ]
\draw[thick,blue,postaction={decorate}] (axis cs:2.3625,0) -- (axis cs:2.3625,-3);
\draw[thick,blue,postaction={decorate}] (axis cs:2.3625,-3) -- (axis cs:-3,-3);
\draw[thick,blue,postaction={decorate}] (axis cs:-3,-3) -- (axis cs:-3,1);
\draw[thick,blue,postaction={decorate}] (axis cs:-3,1) -- (axis cs:1,1);
\draw[thick,blue,postaction={decorate}] (axis cs:1,1) -- (axis cs:1,-2);
\draw[thick,blue,postaction={decorate}] (axis cs:1,-2) -- (axis cs:-2,-2);
\draw[thick,blue,postaction={decorate}] (axis cs:-2,-2) -- (axis cs:-2,0);
\draw[thick,blue,postaction={decorate}] (axis cs:-2,0) -- (axis cs:-1,0);

\fill[black] (axis cs:2.3625,0) circle (1.625pt);

\draw[very thick,blue,postaction={decorate}] (axis cs:-1,0) -- (axis cs:0,0);
\draw[very thick,blue,postaction={decorate}] (axis cs:0,0) -- (axis cs:0,-1);
\draw[very thick,blue,postaction={decorate}] (axis cs:0,-1) -- (axis cs:-1,-1);
\draw[very thick,blue,postaction={decorate}] (axis cs:-1,-1) -- (axis cs:-1,0);
\end{scope}
\end{axis}
\end{tikzpicture}&
		\begin{tikzpicture} [scale=0.55] 
\begin{axis}[
	xmin=-4.5,
	xmax=2.5,
	ymin=-4.5,
	ymax=2.5,
    xtick={-4,-2,0,2},
    ytick={-4,-2,0,2},
	xmajorgrids,
	ymajorgrids
]

\draw[thick] (axis cs:-4.5,0) -- (axis cs:2.5,0);
\draw[thick] (axis cs:0,-4.5) -- (axis cs:0,2.5);

\addplot[domain=-4.5:2.5] (\x,\x);

\draw[red,very thick,fill=white] (axis cs:2.5,-3) circle (1.75pt);

\addplot[red,ultra thick,domain=1.5:2.5] (\x,-2);
\draw[red,very thick,fill=white] (axis cs:1.5,-2) circle (1.75pt);
\draw[red,very thick,fill=red] (axis cs:2.5,-2) circle (1.75pt);

\addplot[red,ultra thick,domain=0.5:1.5] (\x,-1);
\draw[red,very thick,fill=white] (axis cs:0.5,-1) circle (1.75pt);
\draw[red,very thick,fill=red] (axis cs:1.5,-1) circle (1.75pt);

\addplot[red,ultra thick,domain=-0.5:0.5] (\x,0);
\draw[red,very thick,fill=white] (axis cs:-0.5,0) circle (1.75pt);
\draw[red,very thick,fill=red] (axis cs:0.5,0) circle (1.75pt);

\addplot[red,ultra thick,domain=-1.5:-0.5] (\x,1);
\draw[red,very thick,fill=white] (axis cs:-1.5,1) circle (1.75pt);
\draw[red,very thick,fill=red] (axis cs:-0.5,1) circle (1.75pt);

\addplot[red,ultra thick,domain=-2.5:-1.5] (\x,2);
\draw[red,very thick,fill=white] (axis cs:-2.5,2) circle (1.75pt);
\draw[red,very thick,fill=red] (axis cs:-1.5,2) circle (1.75pt);

\begin{scope}[decoration={
    markings,
    mark=at position 0.5 with {\arrow{>}}}
    ]
\draw[thick,blue,postaction={decorate}] (axis cs:2.3625,0) -- (axis cs:2.3625,-2);
\draw[thick,blue,postaction={decorate}] (axis cs:2.3625,-2) -- (axis cs:2,-2);

\draw[thick,blue,postaction={decorate}] (axis cs:1.3625,0) -- (axis cs:1.3625,-1);
\draw[thick,blue,postaction={decorate}] (axis cs:1.3625,-1) -- (axis cs:1,-1);

\draw[thick,blue,postaction={decorate}] (axis cs:0.3625,0) -- (axis cs:0,0);

\draw[thick,blue,postaction={decorate}] (axis cs:-0.3625,0) -- (axis cs:0,0);

\draw[thick,blue,postaction={decorate}] (axis cs:-1.3625,0) -- (axis cs:-1.3625,1);
\draw[thick,blue,postaction={decorate}] (axis cs:-1.3625,1) -- (axis cs:-1,1);

\draw[thick,blue,postaction={decorate}] (axis cs:-2.3625,0) -- (axis cs:-2.3625,2);
\draw[thick,blue,postaction={decorate}] (axis cs:-2.3625,2) -- (axis cs:-2,2);

\fill[blue] (axis cs:0,0) circle (1.625pt);

\fill[black] (axis cs:2.3625,0) circle (1.625pt);
\fill[black] (axis cs:1.3625,0) circle (1.625pt);
\fill[black] (axis cs:0.3625,0) circle (1.625pt);
\fill[black] (axis cs:-0.3625,0) circle (1.625pt);
\fill[black] (axis cs:-1.3625,0) circle (1.625pt);
\fill[black] (axis cs:-2.3625,0) circle (1.625pt);

\draw[very thick,blue,postaction={decorate}] (axis cs:-1,1) -- (axis cs:1,1);
\draw[very thick,blue,postaction={decorate}] (axis cs:1,1) -- (axis cs:1,-1);
\draw[very thick,blue,postaction={decorate}] (axis cs:1,-1) -- (axis cs:-1,-1);
\draw[very thick,blue,postaction={decorate}] (axis cs:-1,-1) -- (axis cs:-1,1);

\draw[very thick,blue,postaction={decorate}] (axis cs:-2,-2) -- (axis cs:-2,2);
\draw[very thick,blue,postaction={decorate}] (axis cs:-2,2) -- (axis cs:2,2);
\draw[very thick,blue,postaction={decorate}] (axis cs:2,2) -- (axis cs:2,-2);
\draw[very thick,blue,postaction={decorate}] (axis cs:2,-2) -- (axis cs:-2,-2);

\end{scope}
\end{axis}
\end{tikzpicture}\\[-0.1375cm]
		{\scriptsize $(\lambda,\mu)=(-0.7,0.5)$ (Theorem \ref{thm:omegalimit} (vi))} & {\scriptsize $(\lambda,\mu)=(-0.7,-0.5)$ (Theorem \ref{thm:omegalimit} (vii))} & {\scriptsize $(\lambda,\mu)=(-1,0.5)$ (Theorem \ref{thm:omegalimit} (viii))}\\[0.15cm]

\begin{tikzpicture} [scale=0.55] 
\begin{axis}[
	xmin=-4.5,
	xmax=2.5,
	ymin=-4.5,
	ymax=2.5,
    xtick={-4,-2,0,2},
    ytick={-4,-2,0,2},
	xmajorgrids,
	ymajorgrids
]

\draw[thick] (axis cs:-4.5,0) -- (axis cs:2.5,0);
\draw[thick] (axis cs:0,-4.5) -- (axis cs:0,2.5);

\addplot[domain=-4.5:2.5] (\x,\x);

\draw[red,very thick,fill=white] (axis cs:2.5,-4) circle (1.75pt);

\addplot[red,ultra thick,domain=1.5:2.5] (\x,-3);
\draw[red,very thick,fill=white] (axis cs:1.5,-3) circle (1.75pt);
\draw[red,very thick,fill=red] (axis cs:2.5,-3) circle (1.75pt);

\addplot[red,ultra thick,domain=0.5:1.5] (\x,-2);
\draw[red,very thick,fill=white] (axis cs:0.5,-2) circle (1.75pt);
\draw[red,very thick,fill=red] (axis cs:1.5,-2) circle (1.75pt);

\addplot[red,ultra thick,domain=-0.5:0.5] (\x,-1);
\draw[red,very thick,fill=white] (axis cs:-0.5,-1) circle (1.75pt);
\draw[red,very thick,fill=red] (axis cs:0.5,-1) circle (1.75pt);

\addplot[red,ultra thick,domain=-1.5:-0.5] (\x,0);
\draw[red,very thick,fill=white] (axis cs:-1.5,0) circle (1.75pt);
\draw[red,very thick,fill=red] (axis cs:-0.5,0) circle (1.75pt);

\addplot[red,ultra thick,domain=-2.5:-1.5] (\x,1);
\draw[red,very thick,fill=white] (axis cs:-2.5,1) circle (1.75pt);
\draw[red,very thick,fill=red] (axis cs:-1.5,1) circle (1.75pt);

\addplot[red,ultra thick,domain=-3.5:-2.5] (\x,2);
\draw[red,very thick,fill=white] (axis cs:-3.5,2) circle (1.75pt);
\draw[red,very thick,fill=red] (axis cs:-2.5,2) circle (1.75pt);

\begin{scope}[decoration={
    markings,
    mark=at position 0.5 with {\arrow{>}}}
    ]
\draw[thick,blue,postaction={decorate}] (axis cs:2.3625,0) -- (axis cs:2.3625,-3);
\draw[thick,blue,postaction={decorate}] (axis cs:2.3625,-3) -- (axis cs:2,-3);

\draw[thick,blue,postaction={decorate}] (axis cs:1.3625,0) -- (axis cs:1.3625,-2);
\draw[thick,blue,postaction={decorate}] (axis cs:1.3625,-2) -- (axis cs:1,-2);

\draw[thick,blue,postaction={decorate}] (axis cs:0.3625,0) -- (axis cs:0.3625,-1);
\draw[thick,blue,postaction={decorate}] (axis cs:0.3625,-1) -- (axis cs:0,-1);

\draw[thick,blue,postaction={decorate}] (axis cs:-1.3625,0) -- (axis cs:-1,0);

\draw[thick,blue,postaction={decorate}] (axis cs:-2.3625,0) -- (axis cs:-2.3625,1);
\draw[thick,blue,postaction={decorate}] (axis cs:-2.3625,1) -- (axis cs:-2,1);

\draw[thick,blue,postaction={decorate}] (axis cs:-3.3625,0) -- (axis cs:-3.3625,2);
\draw[thick,blue,postaction={decorate}] (axis cs:-3.3625,2) -- (axis cs:-3,2);

\fill[black] (axis cs:2.3625,0) circle (1.625pt);
\fill[black] (axis cs:1.3625,0) circle (1.625pt);
\fill[black] (axis cs:0.3625,0) circle (1.625pt);
\fill[black] (axis cs:-1.3625,0) circle (1.625pt);
\fill[black] (axis cs:-2.3625,0) circle (1.625pt);
\fill[black] (axis cs:-3.3625,0) circle (1.625pt);

\draw[very thick,blue,postaction={decorate}] (axis cs:-1,0) -- (axis cs:0,0);
\draw[very thick,blue,postaction={decorate}] (axis cs:0,0) -- (axis cs:0,-1);
\draw[very thick,blue,postaction={decorate}] (axis cs:0,-1) -- (axis cs:-1,-1);
\draw[very thick,blue,postaction={decorate}] (axis cs:-1,-1) -- (axis cs:-1,0);

\draw[very thick,blue,postaction={decorate}] (axis cs:-2,-2) -- (axis cs:-2,1);
\draw[very thick,blue,postaction={decorate}] (axis cs:-2,1) -- (axis cs:1,1);
\draw[very thick,blue,postaction={decorate}] (axis cs:1,1) -- (axis cs:1,-2);
\draw[very thick,blue,postaction={decorate}] (axis cs:1,-2) -- (axis cs:-2,-2);

\draw[very thick,blue,postaction={decorate}] (axis cs:-3,-3) -- (axis cs:-3,2);
\draw[very thick,blue,postaction={decorate}] (axis cs:-3,2) -- (axis cs:2,2);
\draw[very thick,blue,postaction={decorate}] (axis cs:2,2) -- (axis cs:2,-3);
\draw[very thick,blue,postaction={decorate}] (axis cs:2,-3) -- (axis cs:-3,-3);
\end{scope}
\end{axis}
\end{tikzpicture}&\begin{tikzpicture} [scale=0.55] 
\begin{axis}[
	xmin=-4.5,
	xmax=2.5,
	ymin=-4.5,
	ymax=2.5,
    xtick={-4,-2,0,2},
    ytick={-4,-2,0,2},
	xmajorgrids,
	ymajorgrids
]

\draw[thick] (axis cs:-4.5,0) -- (axis cs:2.5,0);
\draw[thick] (axis cs:0,-4.5) -- (axis cs:0,2.5);

\addplot[domain=-4.5:2.5] (\x,\x);
 
\addplot[red,ultra thick,domain=-3.692307692:-2.923076923] (\x,2);
\draw[red,very thick,fill=white] (axis cs:-3.692307692,2) circle (1.75pt);
\draw[red,very thick,fill=red] (axis cs:-2.923076923,2) circle (1.75pt);
 
\addplot[red,ultra thick,domain=-2.923076923:-2.153846154] (\x,1);
\draw[red,very thick,fill=white] (axis cs:-2.923076923,1) circle (1.75pt);
\draw[red,very thick,fill=red] (axis cs:-2.153846154,1) circle (1.75pt);
 
\addplot[red,ultra thick,domain=-2.153846154:-1.384615385] (\x,0);
\draw[red,very thick,fill=white] (axis cs:-2.153846154,0) circle (1.75pt);
\draw[red,very thick,fill=red] (axis cs:-1.384615385,0) circle (1.75pt);
 
\addplot[red,ultra thick,domain=-1.384615385:-.6153846154] (\x,-1);
\draw[red,very thick,fill=white] (axis cs:-1.384615385,-1) circle (1.75pt);
\draw[red,very thick,fill=red] (axis cs:-.6153846154,-1) circle (1.75pt);
 
\addplot[red,ultra thick,domain=-.6153846154:.1538461538] (\x,-2);
\draw[red,very thick,fill=white] (axis cs:-.6153846154,-2) circle (1.75pt);
\draw[red,very thick,fill=red] (axis cs:.1538461538,-2) circle (1.75pt);
 
\addplot[red,ultra thick,domain=.1538461538:.9230769231] (\x,-3);
\draw[red,very thick,fill=white] (axis cs:.1538461538,-3) circle (1.75pt);
\draw[red,very thick,fill=red] (axis cs:.9230769231,-3) circle (1.75pt);
 
\addplot[red,ultra thick,domain=.9230769231:1.692307692] (\x,-4);
\draw[red,very thick,fill=white] (axis cs:.9230769231,-4) circle (1.75pt);
\draw[red,very thick,fill=red] (axis cs:1.692307692,-4) circle (1.75pt);

\begin{scope}[decoration={
    markings,
    mark=at position 0.5 with {\arrow{>}}}
    ]
\draw[thick,blue,postaction={decorate}] (axis cs:-2.5,0) -- (axis cs:-2.5,1);
\draw[thick,blue,postaction={decorate}] (axis cs:-2.5,1) -- (axis cs:1,1);
\draw[thick,blue,postaction={decorate}] (axis cs:1,1) -- (axis cs:1,-4);
\draw[thick,blue,postaction={decorate}] (axis cs:1,-4) -- (axis cs:-4,-4);
\draw[thick,blue,postaction={decorate}] (axis cs:-4,-4) -- (axis cs:-4,2.5);

\draw[thick,blue,postaction={decorate}] (axis cs:0.5,0) -- (axis cs:0.5,-3);
\draw[thick,blue,postaction={decorate}] (axis cs:0.5,-3) -- (axis cs:-3,-3);
\draw[thick,blue,postaction={decorate}] (axis cs:-3,-3) -- (axis cs:-3,2);
\draw[thick,blue,postaction={decorate}] (axis cs:-3,2) -- (axis cs:2,2);
\draw[thick,blue,postaction={decorate}] (axis cs:2,2) -- (axis cs:2,-4.5);

\draw[thick,blue,postaction={decorate}] (axis cs:-0.3,0) -- (axis cs:-0.3,-2);

\draw[thick,blue,postaction={decorate}] (axis cs:-1.25,0) -- (axis cs:-1.25,-1);
\draw[thick,blue,postaction={decorate}] (axis cs:-1.25,-1) -- (axis cs:-1,-1);

\draw[thick,blue,postaction={decorate}] (axis cs:-0.75,0) -- (axis cs:-0.75,-1);
\draw[thick,blue,postaction={decorate}] (axis cs:-0.75,-1) -- (axis cs:-1,-1);

\fill[black] (axis cs:-2.5,0) circle (1.625pt);
\fill[black] (axis cs:0.5,0) circle (1.625pt);
\fill[black] (axis cs:-0.3,0) circle (1.625pt);
\fill[black] (axis cs:-1.25,0) circle (1.625pt);
\fill[black] (axis cs:-0.75,0) circle (1.625pt);

\fill[blue] (axis cs:-1,-1) circle (1.625pt);

\draw[very thick,blue,postaction={decorate}] (axis cs:-2,0) -- (axis cs:0,0);
\draw[very thick,blue,postaction={decorate}] (axis cs:0,0) -- (axis cs:0,-2);
\draw[very thick,blue,postaction={decorate}] (axis cs:0,-2) -- (axis cs:-2,-2);
\draw[very thick,blue,postaction={decorate}] (axis cs:-2,-2) -- (axis cs:-2,0);
\end{scope}
\end{axis}
\end{tikzpicture}&
		\begin{tikzpicture} [scale=0.55] 
\begin{axis}[
	xmin=-4.5,
	xmax=2.5,
	ymin=-4.5,
	ymax=2.5,
    xtick={-4,-2,0,2},
    ytick={-4,-2,0,2},
	xmajorgrids,
	ymajorgrids
]

\draw[thick] (axis cs:-4.5,0) -- (axis cs:2.5,0);
\draw[thick] (axis cs:0,-4.5) -- (axis cs:0,2.5);

\addplot[domain=-4.5:2.5] (\x,\x);
 
\addplot[red,ultra thick,domain=-2.923076923:-2.153846154] (\x,2);
\draw[red,very thick,fill=white] (axis cs:-2.923076923,2) circle (1.75pt);
\draw[red,very thick,fill=red] (axis cs:-2.153846154,2) circle (1.75pt);

\addplot[red,ultra thick,domain=-2.153846154:-1.384615385] (\x,1);
\draw[red,very thick,fill=white] (axis cs:-2.153846154,1) circle (1.75pt);
\draw[red,very thick,fill=red] (axis cs:-1.384615385,1) circle (1.75pt);

\addplot[red,ultra thick,domain=-1.384615385:-.6153846154] (\x,0);
\draw[red,very thick,fill=white] (axis cs:-1.384615385,0) circle (1.75pt);
\draw[red,very thick,fill=red] (axis cs:-.6153846154,0) circle (1.75pt);

\addplot[red,ultra thick,domain=-.6153846154:.1538461538] (\x,-1);
\draw[red,very thick,fill=white] (axis cs:-.6153846154,-1) circle (1.75pt);
\draw[red,very thick,fill=red] (axis cs:.1538461538,-1) circle (1.75pt);

\addplot[red,ultra thick,domain=.1538461538:.9230769231] (\x,-2);
\draw[red,very thick,fill=white] (axis cs:.1538461538,-2) circle (1.75pt);
\draw[red,very thick,fill=red] (axis cs:.9230769231,-2) circle (1.75pt);

\addplot[red,ultra thick,domain=.9230769231:1.692307692] (\x,-3);
\draw[red,very thick,fill=white] (axis cs:.9230769231,-3) circle (1.75pt);
\draw[red,very thick,fill=red] (axis cs:1.692307692,-3) circle (1.75pt);

\addplot[red,ultra thick,domain=1.692307692:2.461538462] (\x,-4);
\draw[red,very thick,fill=white] (axis cs:1.692307692,-4) circle (1.75pt);
\draw[red,very thick,fill=red] (axis cs:2.461538462,-4) circle (1.75pt);

\begin{scope}[decoration={
    markings,
    mark=at position 0.5 with {\arrow{>}}}
    ]
\draw[thick,blue,postaction={decorate}] (axis cs:0.5,0) -- (axis cs:0.5,-2);
\draw[thick,blue,postaction={decorate}] (axis cs:0.5,-2) -- (axis cs:-2,-2);
\draw[thick,blue,postaction={decorate}] (axis cs:-2,-2) -- (axis cs:-2,1);
\draw[thick,blue,postaction={decorate}] (axis cs:-2,1) -- (axis cs:1,1);
\draw[thick,blue,postaction={decorate}] (axis cs:1,1) -- (axis cs:1,-3);
\draw[thick,blue,postaction={decorate}] (axis cs:1,-3) -- (axis cs:-3,-3);
\draw[thick,blue,postaction={decorate}] (axis cs:-3,-3) -- (axis cs:-3,2.5);

\draw[thick,blue,postaction={decorate}] (axis cs:-2.5,0) -- (axis cs:-2.5,2);
\draw[thick,blue,postaction={decorate}] (axis cs:-2.5,2) -- (axis cs:2,2);
\draw[thick,blue,postaction={decorate}] (axis cs:2,2) -- (axis cs:2,-4);
\draw[thick,blue,postaction={decorate}] (axis cs:2,-4) -- (axis cs:-4,-4);
\draw[thick,blue,postaction={decorate}] (axis cs:-4,-4) -- (axis cs:-4,4.5);

\draw[thick,blue,postaction={decorate}] (axis cs:-1.25,0) -- (axis cs:-1,0);

\fill[black] (axis cs:-2.5,0) circle (1.625pt);
\fill[black] (axis cs:-1.25,0) circle (1.625pt);
\fill[black] (axis cs:0.5,0) circle (1.625pt);


\draw[very thick,blue,postaction={decorate}] (axis cs:-1,0) -- (axis cs:0,0);
\draw[very thick,blue,postaction={decorate}] (axis cs:0,0) -- (axis cs:0,-1);
\draw[very thick,blue,postaction={decorate}] (axis cs:0,-1) -- (axis cs:-1,-1);
\draw[very thick,blue,postaction={decorate}] (axis cs:-1,-1) -- (axis cs:-1,0);
\end{scope}
\end{axis}
\end{tikzpicture}\\[-0.1375cm]
		{\scriptsize $(\lambda,\mu)=(-1,-0.5)$ (Theorem \ref{thm:omegalimit} (viii))} & {\scriptsize $(\lambda,\mu)=(-1.3,-1.8)$ (Theorem \ref{thm:omegalimit} (ix))} & {\scriptsize $(\lambda,\mu)=(-1.3,-0.8)$ (Theorem \ref{thm:omegalimit} (viii))}\\[0.15cm]

\begin{tikzpicture} [scale=0.55] 
\begin{axis}[
	xmin=-4.5,
	xmax=2.5,
	ymin=-4.5,
	ymax=2.5,
    xtick={-4,-2,0,2},
    ytick={-4,-2,0,2},
	xmajorgrids,
	ymajorgrids
]

\draw[thick] (axis cs:-4.5,0) -- (axis cs:2.5,0);
\draw[thick] (axis cs:0,-4.5) -- (axis cs:0,2.5);

\addplot[domain=-4.5:2.5] (\x,\x);
 
\addplot[red,ultra thick,domain=-2.521739130: -2.086956522] (\x,2);
\draw[red,very thick,fill=white] (axis cs:-2.521739130,2) circle (1.75pt);
\draw[red,very thick,fill=red] (axis cs:-2.086956522,2) circle (1.75pt);

\addplot[red,ultra thick,domain=-2.086956522:-1.652173913] (\x,1);
\draw[red,very thick,fill=white] (axis cs:-2.086956522,1) circle (1.75pt);
\draw[red,very thick,fill=red] (axis cs:-1.652173913,1) circle (1.75pt);

\addplot[red,ultra thick,domain=-1.652173913:-1.217391304] (\x,0);
\draw[red,very thick,fill=white] (axis cs:-1.652173913,0) circle (1.75pt);
\draw[red,very thick,fill=red] (axis cs:-1.217391304,0) circle (1.75pt);

\addplot[red,ultra thick,domain=-1.217391304:-.7826086957] (\x,-1);
\draw[red,very thick,fill=white] (axis cs:-1.217391304,-1) circle (1.75pt);
\draw[red,very thick,fill=red] (axis cs:-.7826086957,-1) circle (1.75pt);

\addplot[red,ultra thick,domain=-.7826086957:-.3478260870] (\x,-2);
\draw[red,very thick,fill=white] (axis cs:-.7826086957,-2) circle (1.75pt);
\draw[red,very thick,fill=red] (axis cs:-.3478260870,-2) circle (1.75pt);

\addplot[red,ultra thick,domain=-.3478260870:0.8695652174e-1] (\x,-3);
\draw[red,very thick,fill=white] (axis cs:-.3478260870,-3) circle (1.75pt);
\draw[red,very thick,fill=red] (axis cs:0.8695652174e-1,-3) circle (1.75pt);

\addplot[red,ultra thick,domain=0.8695652174e-1:.5217391304] (\x,-4);
\draw[red,very thick,fill=white] (axis cs:0.8695652174e-1,-4) circle (1.75pt);
\draw[red,very thick,fill=red] (axis cs:.5217391304,-4) circle (1.75pt);

\begin{scope}[decoration={
    markings,
    mark=at position 0.5 with {\arrow{>}}}
    ]
\draw[thick,blue,postaction={decorate}] (axis cs:0,0) -- (axis cs:0,-3);
\draw[thick,blue,postaction={decorate}] (axis cs:0,-3) -- (axis cs:-3,-3);
\draw[thick,blue,postaction={decorate}] (axis cs:-3,-3) -- (axis cs:-3,2.5);

\draw[thick,blue,postaction={decorate}] (axis cs:-0.5,0) -- (axis cs:-0.5,-2);
\draw[thick,blue,postaction={decorate}] (axis cs:-0.5,-2) -- (axis cs:-2,-2);
\draw[thick,blue,postaction={decorate}] (axis cs:-2,-2) -- (axis cs:-2,1);
\draw[thick,blue,postaction={decorate}] (axis cs:-2,1) -- (axis cs:1,1);
\draw[thick,blue,postaction={decorate}] (axis cs:1,1) -- (axis cs:1,-4.5);

\draw[thick,blue,postaction={decorate}] (axis cs:-1.125,0) -- (axis cs:-1.125,-1);
\draw[thick,blue,postaction={decorate}] (axis cs:-1.125,-1) -- (axis cs:-1,-1);
\draw[thick,blue,postaction={decorate}] (axis cs:-0.875,0) -- (axis cs:-0.875,-1);
\draw[thick,blue,postaction={decorate}] (axis cs:-0.875,-1) -- (axis cs:-1,-1);

\draw[thick,blue,postaction={decorate}] (axis cs:0.3,0) -- (axis cs:0.3,-4);
\draw[thick,blue,postaction={decorate}] (axis cs:0.3,-4) -- (axis cs:-4,-4);
\draw[thick,blue,postaction={decorate}] (axis cs:-4,-4) -- (axis cs:-4,4.5);

\draw[thick,blue,postaction={decorate}] (axis cs:-2.3,0) -- (axis cs:-2.3,2);
\draw[thick,blue,postaction={decorate}] (axis cs:-2.3,2) -- (axis cs:2,2);
\draw[thick,blue,postaction={decorate}] (axis cs:2,2) -- (axis cs:2,-4.5);

\fill[black] (axis cs:0,0) circle (1.625pt);
\fill[black] (axis cs:-0.5,0) circle (1.625pt);
\fill[black] (axis cs:0.3,0) circle (1.625pt);
\fill[black] (axis cs:-2.3,0) circle (1.625pt);
\fill[black] (axis cs:-1.125,0) circle (1.625pt);
\fill[black] (axis cs:-0.875,0) circle (1.625pt);

\fill[blue] (axis cs:-1,-1) circle (1.625pt);
\end{scope}
\end{axis}
\end{tikzpicture}&\begin{tikzpicture} [scale=0.55] 
\begin{axis}[
	xmin=-4.5,
	xmax=2.5,
	ymin=-4.5,
	ymax=2.5,
    xtick={-4,-2,0,2},
    ytick={-4,-2,0,2},
	xmajorgrids,
	ymajorgrids
]

\draw[thick] (axis cs:-4.5,0) -- (axis cs:2.5,0);
\draw[thick] (axis cs:0,-4.5) -- (axis cs:0,2.5);

\addplot[domain=-4.5:2.5] (\x,\x);

\addplot[red,ultra thick,domain=-2.956521739:-2.521739130] (\x,2);
\draw[red,very thick,fill=white] (axis cs:-2.956521739,2) circle (1.75pt);
\draw[red,very thick,fill=red] (axis cs:-2.521739130,2) circle (1.75pt);

\addplot[red,ultra thick,domain=-2.521739130:-2.086956522] (\x,1);
\draw[red,very thick,fill=white] (axis cs:-2.521739130,1) circle (1.75pt);
\draw[red,very thick,fill=red] (axis cs:-2.086956522,1) circle (1.75pt);

\addplot[red,ultra thick,domain=-2.086956522:-1.652173913] (\x,0);
\draw[red,very thick,fill=white] (axis cs:-2.086956522,0) circle (1.75pt);
\draw[red,very thick,fill=red] (axis cs:-1.652173913,0) circle (1.75pt);

\addplot[red,ultra thick,domain=-1.652173913:-1.217391304] (\x,-1);
\draw[red,very thick,fill=white] (axis cs:-1.652173913,-1) circle (1.75pt);
\draw[red,very thick,fill=red] (axis cs:-1.217391304,-1) circle (1.75pt);

\addplot[red,ultra thick,domain=-1.217391304:-.7826086957] (\x,-2);
\draw[red,very thick,fill=white] (axis cs:-1.217391304,-2) circle (1.75pt);
\draw[red,very thick,fill=red] (axis cs:-.7826086957,-2) circle (1.75pt);

\addplot[red,ultra thick,domain=-.7826086957:-.3478260870] (\x,-3);
\draw[red,very thick,fill=white] (axis cs:-.7826086957,-3) circle (1.75pt);
\draw[red,very thick,fill=red] (axis cs:-.3478260870,-3) circle (1.75pt);

\addplot[red,ultra thick,domain=-.3478260870:0.8695652174e-1] (\x,-4);
\draw[red,very thick,fill=white] (axis cs:-.3478260870,-4) circle (1.75pt);
\draw[red,very thick,fill=red] (axis cs:0.8695652174e-1,-4) circle (1.75pt);

\begin{scope}[decoration={
    markings,
    mark=at position 0.5 with {\arrow{>}}}
    ]
\draw[thick,blue,postaction={decorate}] (axis cs:-1.5,0) -- (axis cs:-1.5,-1);
\draw[thick,blue,postaction={decorate}] (axis cs:-1.5,-1) -- (axis cs:-1,-1);
\draw[thick,blue,postaction={decorate}] (axis cs:-1,-1) -- (axis cs:-1,-2);
\draw[thick,blue,postaction={decorate}] (axis cs:-1,-2) -- (axis cs:-2,-2);
\draw[thick,blue,postaction={decorate}] (axis cs:-2,-2) -- (axis cs:-2,0);
\draw[thick,blue,postaction={decorate}] (axis cs:-2,0) -- (axis cs:0,0);
\draw[thick,blue,postaction={decorate}] (axis cs:0,0) -- (axis cs:0,-4);
\draw[thick,blue,postaction={decorate}] (axis cs:0,-4) -- (axis cs:-4,-4);
\draw[thick,blue,postaction={decorate}] (axis cs:-4,-4) -- (axis cs:-4,2.5);

\draw[thick,blue,postaction={decorate}] (axis cs:-2.35,0) -- (axis cs:-2.35,1);
\draw[thick,blue,postaction={decorate}] (axis cs:-2.35,1) -- (axis cs:1,1);
\draw[thick,blue,postaction={decorate}] (axis cs:1,1) -- (axis cs:1,-4.5);

\draw[thick,blue,postaction={decorate}] (axis cs:-2.8,0) -- (axis cs:-2.8,2);
\draw[thick,blue,postaction={decorate}] (axis cs:-2.8,2) -- (axis cs:2,2);
\draw[thick,blue,postaction={decorate}] (axis cs:2,2) -- (axis cs:2,-4.5);

\draw[thick,blue,postaction={decorate}] (axis cs:-0.5,0) -- (axis cs:-0.5,-3);
\draw[thick,blue,postaction={decorate}] (axis cs:-0.5,-3) -- (axis cs:-3,-3);
\draw[thick,blue,postaction={decorate}] (axis cs:-3,-3) -- (axis cs:-3,2.5);

\fill[black] (axis cs:-1.5,0) circle (1.625pt);
\fill[black] (axis cs:-2.35,0) circle (1.625pt);
\fill[black] (axis cs:-2.8,0) circle (1.625pt);
\fill[black] (axis cs:-0.5,0) circle (1.625pt);

\end{scope}
\end{axis}
\end{tikzpicture}\\[-0.1375cm]
		{\scriptsize $(\lambda,\mu)=(-2.3,-2.8)$ (Theorem \ref{thm:omegalimit} (ix))} & {\scriptsize $(\lambda,\mu)=(-2.3,-3.8)$ (Theorem \ref{thm:omegalimit} (viii))}
\end{tabular}		
		\caption{\label{fig:cobweb}Cobweb diagrams of the map \eqref{eq:maps} for various pairs of parameter values representing different cases considered in Theorem \ref{thm:omegalimit}.}
\end{figure}

\section{Proofs of Theorems \ref{thm:FP} and \ref{thm:2cycles}}\label{sec:FP}

In this section, we prove Theorems \ref{thm:FP} and \ref{thm:2cycles} concerning the fixed points and 2-cycles of $f$. Here, as also in section \ref{sec:LP}, the following basic properties of the floor and ceiling functions will be used without explicit mention.
\begin{itemize}[leftmargin=0.5cm]
\item The floor and ceiling functions are monotonically non-decreasing, i.e., for every $x,y\in\mathbb{R}$ we have that $x\leqslant y$ implies $\left\lfloor x\right\rfloor \leqslant\left\lfloor y\right\rfloor$ and $\left\lceil x\right\rceil\leqslant\left\lceil y\right\rceil$.
\item For every $x\in\mathbb{R}$ we have $x-1<\left\lfloor x\right\rfloor\leqslant x$ and $x\leqslant\left\lceil x\right\rceil<x+1$.
\item For every $x,n\in\mathbb{R}$ we have $\left\lfloor x\right\rfloor = n$ if and only if $n\in\mathbb{Z}$ and $n\leqslant x<n+1$, and $\left\lceil x\right\rceil=n$ if and only if $n\in\mathbb{Z}$ and $n-1<x\leqslant n$.
\item For every $x\in\mathbb{R}$ and $n\in\mathbb{Z}$ we have $\left\lfloor x+n\right\rfloor=\left\lfloor x\right\rfloor+n$ and $\left\lceil x+n\right\rceil=\left\lceil x\right\rceil+n$.
\end{itemize}\bigskip

\begin{proofFP}
We have
\begin{align*}
\Fix(f)&=\left\{x\in\Z:0\leqslant(\lambda-1)x+\mu<1\right\}\\
&=\begin{cases}
\mathbb{Z},&\text{if }\lambda=1\text{ and }\mu\in[0,1);\\
\varnothing,&\text{if }\lambda=1\text{ and }\mu\in(-\infty,0)\cup[1,\infty);\\
\left\{x\in\Z:-\frac{\mu}{\lambda-1}\leqslant x <-\frac{\mu-1}{\lambda-1}\right\},&\text{if }\lambda>1;\\
\left\{x\in\Z:-\frac{\mu-1}{\lambda-1}< x \leqslant-\frac{\mu}{\lambda-1}\right\},&\text{if }\lambda<1,
\end{cases}
\end{align*}
and hence the theorem.
\end{proofFP}\bigskip

\begin{proof2cycles}
Since the case $\lambda=-1$ is straightforward, let us assume $\lambda\neq -1$. We seek all $\{x,y\}\subseteq\mathbb{Z}$, where $x<y$, satisfying
\begin{equation}\label{eq:proof2cyclessystem}
\left\{\begin{array}{rcl}
\left\lfloor \lambda x+\mu\right\rfloor&=&y,\\
\left\lfloor \lambda y+\mu\right\rfloor&=&x.
\end{array}\right.
\end{equation}
Subtracting gives an equation which necessitates
\begin{equation}\label{eq:proof2cyclesineq}
-1<(\lambda+1)(y-x)<1.
\end{equation}

Suppose $\lambda>-1$. Then \eqref{eq:proof2cyclesineq} is equivalent to $0<y-x<\frac{1}{\lambda+1}$. If $\frac{1}{\lambda+1}\leqslant 1$, i.e., $\lambda\geqslant 0$, then there is no 2-cycle. Otherwise, since $\left\lceil\frac{1}{\lambda+1}\right\rceil-1<\frac{1}{\lambda+1}\leqslant \left\lceil\frac{1}{\lambda+1}\right\rceil$, then $y=x+k$, where $k\in\left\{1,\ldots,\left\lceil\frac{1}{\lambda+1}\right\rceil-1\right\}$. Substituting this into \eqref{eq:proof2cyclessystem} gives
\begin{equation}\label{eq:proof2cyclessystemineq0}
\left\{\begin{array}{rcl}
\left\lfloor\lambda x+\mu\right\rfloor&=&x+k,\\
\left\lfloor\lambda (x+k)+\mu\right\rfloor&=&x,
\end{array}\right.
\end{equation}
which is equivalent to
\begin{equation}\label{eq:proof2cyclessystemineq1}
\left\{\begin{array}{rcl}
\frac{k+1-\mu}{\lambda-1}<x\leqslant\frac{k-\mu}{\lambda-1},\\
\frac{-\lambda k-\mu+1}{\lambda-1}<x\leqslant\frac{-\lambda k-\mu}{\lambda-1}.
\end{array}\right.
\end{equation}
The conditions $0<k\leqslant \left\lceil\frac{1}{\lambda+1}\right\rceil-1<\frac{1}{\lambda+1}$ and $\lambda>-1$ imply
$$\frac{k+1-\mu}{\lambda-1}<\frac{-\lambda k-\mu+1}{\lambda-1}<\frac{k-\mu}{\lambda-1}<\frac{-\lambda k-\mu}{\lambda-1},$$
and so \eqref{eq:proof2cyclessystemineq1} is equivalent to
$$x\in\mathbb{Z}\cap\left(\frac{-\lambda k-\mu+1}{\lambda-1},\frac{k-\mu}{\lambda-1}\right]=\left\{\left\lfloor\frac{-\lambda k-\mu+1}{\lambda-1}\right\rfloor+1,\ldots,\left\lfloor\frac{k-\mu}{\lambda-1}\right\rfloor\right\}.$$
Thus, the set of all $2$-cycles of $f$ is
$$\bigcup_{k=1}^{\left\lceil\frac{1}{\lambda+1}\right\rceil-1}\left\{\left\{x,x+k\right\}:x\in\left\{\left\lfloor\frac{-\lambda k-\mu+1}{\lambda-1}\right\rfloor+1,\ldots,\left\lfloor\frac{k-\mu}{\lambda-1}\right\rfloor\right\}\right\}.$$

In the case $\lambda<-1$, \eqref{eq:proof2cyclesineq} is equivalent to $0<y-x<-\frac{1}{\lambda+1}$. As before, if $-\frac{1}{\lambda+1}\leqslant 1$, i.e., $\lambda\leqslant-2$, then there is no $2$-cycle. Otherwise, since ${-\left\lfloor\frac{1}{\lambda+1}\right\rfloor-1}<-\frac{1}{\lambda+1}\leqslant -\left\lfloor\frac{1}{\lambda+1}\right\rfloor$, then $y=x+k$, where $k\in\left\{1,\ldots,-\left\lfloor\frac{1}{\lambda+1}\right\rfloor-1\right\}$. Substituting this into \eqref{eq:proof2cyclessystem} gives \eqref{eq:proof2cyclessystemineq0}, and hence \eqref{eq:proof2cyclessystemineq1}. The conditions $0<k\leqslant -\left\lfloor\frac{1}{\lambda+1}\right\rfloor-1<-\frac{1}{\lambda+1}$ and $\lambda<-1$ imply
$$\frac{-\lambda k-\mu+1}{\lambda-1}<\frac{k+1-\mu}{\lambda-1}<\frac{-\lambda k-\mu}{\lambda-1}<\frac{k-\mu}{\lambda-1}.$$
The rest is analogous.
\end{proof2cycles}\bigskip

\section{Proof of Theorem \ref{thm:omegalimit}}\label{sec:LP}

In this section, we establish Theorem \ref{thm:omegalimit} via several propositions: Propositions \ref{prop1}--\ref{prop4} and \ref{prop5}--\ref{prop9}, considering separately the cases $\lambda\geqslant 0$ in which $f$ is monotonically non-decreasing (subsection \ref{subsec1}) and $\lambda<0$ in which $f$ is monotonically non-increasing (subsection \ref{subsec2}).

\subsection{The case $\lambda\geqslant 0$}\label{subsec1}

In the case $\lambda\geqslant 0$, $f$ is monotonically non-decreasing, and so $f$ may possess any number of fixed points. Let us deal with the two easiest subcases first.

\subsubsection{The subcases $\lambda=0$ and $\lambda=1$}

The subcase $\lambda=0$ is trivial, and is given by the following proposition.\medskip

\begin{proposition}\label{prop1}
If $\lambda=0$, then for every $x\in\mathbb{R}$ we have
$\omega_f(x)=\left\{\left\lfloor \mu\right\rfloor\right\}$.
\end{proposition}\medskip

\noindent We now turn to the subcase $\lambda=1$.\medskip

\begin{proposition}\label{prop2}
If $\lambda=1$, then the following holds.
\begin{itemize}[leftmargin=0.5cm]
\item If $\mu\geqslant 1$, then for every $x\in\mathbb{R}$ we have $\omega_f(x)=\{\infty\}$.
\item If $\mu<0$, then for every $x\in\mathbb{R}$ we have $\omega_f(x)=\{-\infty\}$.
\item If $\mu\in[0,1)$, then for every $x\in\mathbb{R}$ we have $\omega_f(x)=\left\{\left\lfloor x+\mu\right\rfloor\right\}$.
\end{itemize}
\end{proposition}
\begin{proof}
Suppose $\lambda=1$. If $\mu\geqslant1$ or $\mu<0$, then $\Fix(f)=\varnothing$, and we have $f(x)>x$ for every $x\in\mathbb{R}$ in the former case and $f(x)<x$ for every $x\in\mathbb{R}$ in the latter. If $0\leqslant\mu<1$, then for every $x\in\mathbb{R}$, $f(x)=\left\lfloor x+\mu\right\rfloor\in\Fix(f)$. The proposition follows.
\end{proof}\medskip

\subsubsection{The subcase $\lambda>1$} The subcase $\lambda>1$ is given by the following proposition.\medskip

\begin{proposition}\label{prop3}
If $\lambda>1$ and $\left\lceil-\frac{\mu}{\lambda-1}\right\rceil\leqslant\left\lceil-\frac{\mu-1}{\lambda-1}\right\rceil-1$, then for every $x\in\mathbb{R}$ we have
\begin{equation}\label{eq:prop3eq1}
\omega_f(x)=\begin{cases}
\{\infty\},&\text{if }x\geqslant\frac{1}{\lambda}\left(\left\lceil-\frac{\mu-1}{\lambda-1}\right\rceil-\mu\right);\\
\{-\infty\},&\text{if }x<\frac{1}{\lambda}\left(\left\lceil-\frac{\mu}{\lambda-1}\right\rceil-\mu\right);\\
\{f(x)\},&\text{if }\frac{1}{\lambda}\left(\left\lceil-\frac{\mu}{\lambda-1}\right\rceil-\mu\right)\leqslant x<\frac{1}{\lambda}\left(\left\lceil-\frac{\mu-1}{\lambda-1}\right\rceil-\mu\right).
\end{cases}
\end{equation}
If $\lambda>1$ and $\left\lceil-\frac{\mu}{\lambda-1}\right\rceil>\left\lceil-\frac{\mu-1}{\lambda-1}\right\rceil-1$, then for every $x\in\mathbb{R}$ we have
\begin{equation}\label{eq:prop3eq2}
\omega_f(x)=\begin{cases}
\{\infty\},&\text{if }x\geqslant\frac{1}{\lambda}\left(\left\lfloor-\frac{\mu}{\lambda-1}\right\rfloor-\mu+1\right);\\
\{-\infty\},&\text{if }x<\frac{1}{\lambda}\left(\left\lfloor-\frac{\mu}{\lambda-1}\right\rfloor-\mu+1\right).
\end{cases}
\end{equation}
\end{proposition}
\begin{proof}
Suppose $\lambda>1$. Let $x\in\mathbb{R}$. If $x\geqslant-\frac{\mu-1}{\lambda-1}$ or $x<-\frac{\mu}{\lambda-1}$, then $x$ is not a fixed point, and we have $f(x)>x$ in the former case and $f(x)<x$ in the latter. Consequently, if $x\geqslant-\frac{\mu-1}{\lambda-1}$ then $\omega_f(x)=\{\infty\}$, and if $x<-\frac{\mu}{\lambda-1}$ then $\omega_f(x)=\{-\infty\}$.

Now suppose $-\frac{\mu}{\lambda-1}\leqslant x<-\frac{\mu-1}{\lambda-1}$. Then $\left\lfloor x\right\rfloor \leqslant  \lambda x+\mu < \left\lfloor x\right\rfloor+2$, and so
$$f(x)=\begin{cases}
\left\lfloor x\right\rfloor+1,&\text{if }\lambda x+\mu\geqslant\left\lfloor x\right\rfloor+1;\\
\left\lfloor x\right\rfloor,&\text{if }\lambda x+\mu<\left\lfloor x\right\rfloor+1.
\end{cases}$$
Therefore, for every $n\in\mathbb{Z}$, if $n\leqslant x<n+1$, then
$$f(x)=\begin{cases}
\left\lfloor x\right\rfloor+1>x,&\text{if }x\geqslant\frac{1}{\lambda}\left(n-\mu+1\right);\\
\left\lfloor x\right\rfloor\leqslant x,&\text{if }x<\frac{1}{\lambda}\left(n-\mu+1\right).
\end{cases}$$
In the case $\left\lceil-\frac{\mu}{\lambda-1}\right\rceil\leqslant\left\lceil-\frac{\mu-1}{\lambda-1}\right\rceil-1$, if $\frac{1}{\lambda}\left(\left\lceil-\frac{\mu}{\lambda-1}\right\rceil-\mu\right)\leqslant x<\frac{1}{\lambda}\left(\left\lceil-\frac{\mu-1}{\lambda-1}\right\rceil-\mu\right)$ then $f(x)\in\Fix(f)=\left\{\left\lceil -\frac{\mu}{\lambda-1}\right\rceil,\left\lceil-\frac{\mu}{\lambda-1}\right\rceil+1,\ldots,\left\lceil-\frac{\mu-1}{\lambda-1}\right\rceil-1\right\}$, if $x\geqslant\frac{1}{\lambda}\left(\left\lceil-\frac{\mu-1}{\lambda-1}\right\rceil-\mu\right)$ then $f(x)>x$, and if $x<\frac{1}{\lambda}\left(\left\lceil-\frac{\mu}{\lambda-1}\right\rceil-\mu\right)$ then $f(x)<x$. The formula \eqref{eq:prop3eq1} follows. In the opposite case, $\Fix(f)=\varnothing$, and since $\left\lfloor-\frac{\mu}{\lambda-1}\right\rfloor \leqslant x<\left\lfloor-\frac{\mu}{\lambda-1}\right\rfloor+1$, if $x\geqslant \frac{1}{\lambda}\left(\left\lfloor-\frac{\mu}{\lambda-1}\right\rfloor-\mu+1\right)$ then $f(x)>x$, and if $x<\frac{1}{\lambda}\left(\left\lfloor-\frac{\mu}{\lambda-1}\right\rfloor-\mu+1\right)$ then $f(x)<x$. The formula \eqref{eq:prop3eq2} follows.
\end{proof}\medskip

\subsubsection{The subcase $0<\lambda<1$} The subcase $0<\lambda<1$ is given by the following proposition.\medskip

\begin{proposition}\label{prop4}
If $0<\lambda<1$, then for every $x\in\mathbb{R}$ we have
$$\omega_f(x)=\begin{cases}
\left\{\left\lfloor-\frac{\mu-1}{\lambda-1}\right\rfloor+1\right\},&\text{if }x\leqslant-\frac{\mu-1}{\lambda-1};\\
\left\{\left\lfloor-\frac{\mu}{\lambda-1}\right\rfloor\right\},&\text{if }x>-\frac{\mu}{\lambda-1};\\
\{f(x)\},&\text{if }{-\frac{\mu-1}{\lambda-1}<x\leqslant-\frac{\mu}{\lambda-1}}.
\end{cases}$$
\end{proposition}
\begin{proof}
Suppose $0<\lambda<1$. Then $\left\lfloor-\frac{\mu-1}{\lambda-1}\right\rfloor+1\leqslant\left\lfloor-\frac{\mu}{\lambda-1}\right\rfloor$, which implies that $$\Fix(f)=\left\{\left\lfloor-\frac{\mu-1}{\lambda-1}\right\rfloor+1,\left\lfloor-\frac{\mu-1}{\lambda-1}\right\rfloor+2,\ldots,\left\lfloor-\frac{\mu}{\lambda-1}\right\rfloor\right\}\neq\varnothing.$$
Let $x\in\mathbb{R}$. If $x\leqslant-\frac{\mu-1}{\lambda-1}$ then $f(x)\in \left(x, \left\lfloor-\frac{\mu-1}{\lambda-1}\right\rfloor+1\right]\cap\mathbb{Z}$, while if $x>-\frac{\mu}{\lambda-1}$ then $f(x)\in \left[\left\lfloor-\frac{\mu}{\lambda-1}\right\rfloor,x\right)\cap\mathbb{Z}$. It follows that if $x\leqslant-\frac{\mu-1}{\lambda-1}$ then $\omega_f(x)=\left\{\left\lfloor-\frac{\mu-1}{\lambda-1}\right\rfloor+1\right\}$, while if $x>-\frac{\mu}{\lambda-1}$ then $\omega_f(x)=\left\{\left\lfloor-\frac{\mu}{\lambda-1}\right\rfloor\right\}$.

Now suppose $-\frac{\mu-1}{\lambda-1}<x\leqslant-\frac{\mu}{\lambda-1}$. Then $\left\lfloor x\right\rfloor \leqslant  \lambda x+\mu < \left\lfloor x\right\rfloor+2$, and so
$$f(x)=\begin{cases}
\left\lfloor x\right\rfloor+1,&\text{if }\lambda x+\mu\geqslant\left\lfloor x\right\rfloor+1;\\
\left\lfloor x\right\rfloor,&\text{if }\lambda x+\mu<\left\lfloor x\right\rfloor+1.
\end{cases}$$
Therefore, for every $n\in\mathbb{Z}$, if $n\leqslant x<n+1$, then
\begin{equation}\label{eq:prop4proof}
f(x)=\begin{cases}
\left\lfloor x\right\rfloor+1>x,&\text{if }x\geqslant\frac{1}{\lambda}\left(n-\mu+1\right);\\
\left\lfloor x\right\rfloor\leqslant x,&\text{if }x<\frac{1}{\lambda}\left(n-\mu+1\right).
\end{cases}
\end{equation}
Since $\frac{1}{\lambda}\left(\left\lfloor-\frac{\mu-1}{\lambda-1}\right\rfloor-\mu+1\right)\leqslant x<\frac{1}{\lambda}\left(\left\lfloor-\frac{\mu}{\lambda-1}\right\rfloor-\mu+1\right)$, then $f(x)\in\Fix(f)$. The proposition follows.
\end{proof}\medskip

\subsection{The case $\lambda<0$}\label{subsec2}

Now suppose that $\lambda<0$. For every $n\in\mathbb{Z}$, we define the subinterval
$$I_n:=\left(\frac{1}{\lambda}\left(\left\lfloor-\frac{\mu}{\lambda-1}\right\rfloor-\mu+n+1\right),\frac{1}{\lambda}\left(\left\lfloor-\frac{\mu}{\lambda-1}\right\rfloor-\mu+n\right)\right].$$
Notice that these subintervals partition $\mathbb{R}$. Moreover, we have the following lemma which is straightforward.\medskip

\begin{lemma}\label{lemma1}
Let $\lambda<0$. For every $n\in\mathbb{Z}$ we have $f(x)=\left\lfloor-\frac{\mu}{\lambda-1}\right\rfloor+n$ if and only if $x\in I_n$.
\end{lemma}\medskip

\noindent Since $\lambda<0$, then $f$ is monotonically non-increasing, and so $\left|\Fix(f)\right|\in\{0,1\}$.\medskip

\subsubsection{The subcase $\left\lfloor-\frac{\mu-1}{\lambda-1}\right\rfloor+1\leqslant\left\lfloor-\frac{\mu}{\lambda-1}\right\rfloor$}

In this subcase, we have $\left|\Fix(f)\right|=1$, and $\left\lfloor-\frac{\mu-1}{\lambda-1}\right\rfloor+1=\left\lfloor-\frac{\mu}{\lambda-1}\right\rfloor$ is the unique fixed point of $f$. By Lemma \ref{lemma1}, this implies the following proposition.\medskip

\begin{proposition}\label{prop5}
Suppose $\left\lfloor-\frac{\mu-1}{\lambda-1}\right\rfloor+1\leqslant\left\lfloor-\frac{\mu}{\lambda-1}\right\rfloor$. If $x\in I_0$, then $\omega_f(x)=\left\{\left\lfloor-\frac{\mu}{\lambda-1}\right\rfloor\right\}$.
\end{proposition}\medskip

Now suppose that $x\in I_n$, where $n\in\mathbb{Z}\smallsetminus\{0\}$. Then Lemma \ref{lemma1} and the fact that $\left\lfloor-\frac{\mu}{\lambda-1}\right\rfloor$ is a fixed point of $f$ imply
$$\frac{1}{\lambda}\left(\left\lfloor-\frac{\mu}{\lambda-1}\right\rfloor-\mu+1\right)+n<f(x)\leqslant \frac{1}{\lambda}\left(\left\lfloor-\frac{\mu}{\lambda-1}\right\rfloor-\mu\right)+n.$$\smallskip

\paragraph{\textit{The subsubcase $\lambda\leqslant-1$}} In this subsubcase, if $n\leqslant-1$ then $$f(x)\leqslant\frac{1}{\lambda}\left(\left\lfloor-\frac{\mu}{\lambda-1}\right\rfloor-\mu+(-n)\right),\quad\text{i.e.},\quad f(x)\in I_{-n}\cup I_{-n+1}\cup I_{-n+2}\cup\cdots,$$
whereas if $n\geqslant1$ then $$f(x)>\frac{1}{\lambda}\left(\left\lfloor-\frac{\mu}{\lambda-1}\right\rfloor-\mu+(-n)+1\right),\quad\text{i.e.},\quad f(x)\in I_{-n}\cup I_{-n-1}\cup I_{-n-2}\cup\cdots.$$
It follows that for every $n\in\mathbb{N}$, the sets
$$I_{-n}\cup I_{-n-1}\cup I_{-n-2}\cup\cdots\qquad\text{and}\qquad I_{n}\cup I_{n+1}\cup I_{n+2}\cup\cdots$$
are invariant under $f^2$. This leads us to the following proposition.\medskip

\begin{proposition}\label{prop6}
Suppose $\left\lfloor-\frac{\mu-1}{\lambda-1}\right\rfloor+1\leqslant\left\lfloor-\frac{\mu}{\lambda-1}\right\rfloor$ and $\lambda\leqslant -1$. If $x\in I_n$, where $n\in\mathbb{Z}\smallsetminus\{0\}$, then $\omega_f(x)$ is either a $2$-cycle or $\{-\infty,\infty\}$.
\end{proposition}
\begin{proof}
Suppose $\left\lfloor-\frac{\mu-1}{\lambda-1}\right\rfloor+1\leqslant\left\lfloor-\frac{\mu}{\lambda-1}\right\rfloor$ and $\lambda\leqslant -1$. Let $x\in I_n$, where $n\in\mathbb{Z}\smallsetminus\{0\}$. Suppose $n\geqslant 1$. Then $f^2(x)\in I_n\cup I_{n+1}\cup I_{n+2}\cup\cdots$. If $f^2(x)\in I_n$, then $f^3(x)=f(x)$, and so $\omega_f(x)$ is a 2-cycle. Otherwise, there exists a unique $k_1\in\mathbb{N}$ such that $f^2(x)\in I_{n+k_1}$. Then $f^4(x)\in I_{n+k_1}\cup I_{n+k_1+1}\cup I_{n+k_1+2}\cup\cdots$. If $f^4(x)\in I_{n+k_1}$, then $f^5(x)=f^3(x)$, and so $\omega_f(x)$ is a 2-cycle. Otherwise, there exists a unique $k_2\in\mathbb{N}$ with $k_1<k_2$ such that $f^4(x)\in I_{n+k_2}$. Then $f^6(x)\in I_{n+k_2}\cup I_{n+k_2+1}\cup I_{n+k_2+2}\cup\cdots$. Continuing this, either at some point one concludes that $\omega_f(x)$ is a 2-cycle or there exists an infinite sequence $k_1<k_2<k_3<\cdots$ of positive integers such that $f^2(x)\in I_{n+k_1}$, $f^4(x)\in I_{n+k_2}$, $f^6(x)\in I_{n+k_3}$, \ldots, which implies that $\omega_f(x)=\{-\infty,\infty\}$. The case $n\leqslant-1$ is analogous.
\end{proof}\medskip

\paragraph{\textit{The subsubcase $-1<\lambda<0$}} In this subsubcase, if $n\leqslant-1$ then $$f(x)>\frac{1}{\lambda}\left(\left\lfloor-\frac{\mu}{\lambda-1}\right\rfloor-\mu+(-n)+1\right),\,\,\,\text{i.e.},\,\,\,f(x)\in I_{-n}\cup I_{-n-1}\cup I_{-n-2}\cup\cdots\cup I_1\cup I_0,$$
whereas if $n\geqslant1$ then
$$f(x)\leqslant\frac{1}{\lambda}\left(\left\lfloor-\frac{\mu}{\lambda-1}\right\rfloor-\mu+(-n)\right),\,\,\,\text{i.e.},\,\,\,f(x)\in I_{-n}\cup I_{-n+1}\cup I_{-n+2}\cup\cdots\cup I_{-1}\cup I_0.$$
It follows that for every $n\in\mathbb{N}$, the sets
$$I_{-n}\cup I_{-n+1}\cup I_{-n+2}\cup\cdots\cup I_{-1}\cup I_0\qquad\text{and}\qquad I_n\cup I_{n-1}\cup I_{n-2}\cup\cdots\cup I_1\cup I_0$$
are invariant under $f^2$. This leads us to the following proposition.\medskip

\begin{proposition}\label{prop7}
Suppose $\left\lfloor-\frac{\mu-1}{\lambda-1}\right\rfloor+1\leqslant\left\lfloor-\frac{\mu}{\lambda-1}\right\rfloor$ and $-1<\lambda<0$. If $x\leqslant\frac{1}{\lambda}\left(\left\lfloor-\frac{\mu}{\lambda-1}\right\rfloor-\mu+1\right)$ or $x>\frac{1}{\lambda}\left(\left\lfloor-\frac{\mu}{\lambda-1}\right\rfloor-\mu\right)$, then $\omega_f(x)$ is either a $2$-cycle or $\left\{\left\lfloor-\frac{\mu}{\lambda-1}\right\rfloor\right\}$.
\end{proposition}
\begin{proof}
Suppose $\left\lfloor-\frac{\mu-1}{\lambda-1}\right\rfloor+1\leqslant\left\lfloor-\frac{\mu}{\lambda-1}\right\rfloor$ and $-1<\lambda<0$. Let $x\in I_n$, where $n\in\mathbb{Z}\smallsetminus\{0\}$. Suppose $n\geqslant 1$. Then $f^2(x)\in I_n\cup I_{n-1}\cup I_{n-2}\cup\cdots\cup I_1\cup I_0$. If $f^2(x)\in I_n$ then $f^3(x)=f(x)$, and so $\omega_f(x)$ is a 2-cycle. If $f^2(x)\in I_0$, then $f^2(x)=\left\lfloor-\frac{\mu}{\lambda-1}\right\rfloor$, and so $\omega_f(x)=\left\{\left\lfloor-\frac{\mu}{\lambda-1}\right\rfloor\right\}$. Otherwise, there exists a unique $k_1\in\{1,\ldots,n-1\}$ such that $f^2(x)\in I_{n-k_1}$. Then $f^4(x)\in I_{n-k_1}\cup I_{n-k_1-1}\cup I_{n-k_1-2}\cup\cdots\cup I_1\cup I_0$. If $f^4(x)\in I_{n-k_1}$, then $f^5(x)=f^3(x)$, and so $\omega_f(x)$ is a 2-cycle. If $f^4(x)\in I_0$, then $f^4(x)=\left\lfloor-\frac{\mu}{\lambda-1}\right\rfloor$, and so $\omega_f(x)=\left\{\left\lfloor-\frac{\mu}{\lambda-1}\right\rfloor\right\}$. Otherwise, there exists a unique $k_2\in\{1,\ldots,n-1\}$ with $k_1<k_2$ such that $f^4(x)\in I_{n-k_2}$. Then $f^6(x)\in I_{n-k_2}\cup I_{n-k_2-1}\cup I_{n-k_2-2}\cup\cdots\cup I_1\cup I_0$. Continuing this, since the sequence $k_1<k_2<k_3<\cdots$ of elements of $\{1,\ldots,n-1\}$ must be finite, the process terminates with the conclusion of either $\omega_f(x)$ being a 2-cycle or $\omega_f(x)=\left\{\left\lfloor-\frac{\mu}{\lambda-1}\right\rfloor\right\}$. The case $n\geqslant-1$ is analogous.
\end{proof}\medskip

\subsubsection{The subcase $\left\lfloor-\frac{\mu-1}{\lambda-1}\right\rfloor+1>\left\lfloor-\frac{\mu}{\lambda-1}\right\rfloor$}

In this subcase, we have $\left|\Fix(f)\right|=0$, and
\begin{equation}\label{eq:equality}
\left\lfloor-\frac{\mu-1}{\lambda-1}\right\rfloor=\left\lfloor-\frac{\mu}{\lambda-1}\right\rfloor.
\end{equation}
Indeed, if $\left\lfloor-\frac{\mu-1}{\lambda-1}\right\rfloor>\left\lfloor-\frac{\mu}{\lambda-1}\right\rfloor$ then $-\frac{\mu-1}{\lambda-1}>-\frac{\mu}{\lambda-1}$, which is a contradiction since $\lambda<0$. Using \eqref{eq:equality}, one shows that
\begin{equation}\label{eq:ineqsecondcase}
\frac{1}{\lambda}\left(\left\lfloor-\frac{\mu}{\lambda-1}\right\rfloor-\mu+1\right)-1<\left\lfloor-\frac{\mu}{\lambda-1}\right\rfloor\leqslant\frac{1}{\lambda}\left(\left\lfloor-\frac{\mu}{\lambda-1}\right\rfloor-\mu+1\right).
\end{equation}

Now let $x\in\mathbb{R}$. Suppose that $x\in I_n$, where $n\in\mathbb{Z}$. Then Lemma \ref{lemma1} and \eqref{eq:ineqsecondcase} imply
$$\frac{1}{\lambda}\left(\left\lfloor-\frac{\mu}{\lambda-1}\right\rfloor-\mu+1\right)-1+n<f(x)\leqslant\frac{1}{\lambda}\left(\left\lfloor-\frac{\mu}{\lambda-1}\right\rfloor-\mu+1\right)+n.$$\smallskip

\paragraph{\textit{The subsubcase $\lambda\leqslant-1$}} In this subsubcase, if $n\leqslant0$ then $$f(x)\leqslant\frac{1}{\lambda}\left(\left\lfloor-\frac{\mu}{\lambda-1}\right\rfloor-\mu+(-n+1)\right),\quad\text{i.e.},\quad f(x)\in I_{-n+1}\cup I_{-n+2}\cup I_{-n+3}\cup\cdots,$$
whereas if $n\geqslant1$ then $$f(x)>\frac{1}{\lambda}\left(\left\lfloor-\frac{\mu}{\lambda-1}\right\rfloor-\mu+(-n+1)+1\right),\quad\text{i.e.},\quad f(x)\in I_{-n+1}\cup I_{-n}\cup I_{-n-1}\cup\cdots.$$
It follows that for every $n\in\mathbb{N}$, the sets
$$I_{-n+1}\cup I_{-n}\cup I_{-n-1}\cup\cdots\qquad\text{and}\qquad I_{n}\cup I_{n+1}\cup I_{n+2}\cup\cdots$$
are invariant under $f^2$. This leads us to the following proposition, which can be proved in the same way as Proposition \ref{prop6}.\medskip

\begin{proposition}\label{prop8}
Suppose $\left\lfloor-\frac{\mu-1}{\lambda-1}\right\rfloor+1>\left\lfloor-\frac{\mu}{\lambda-1}\right\rfloor$. If $\lambda\leqslant-1$, then for every $x\in\mathbb{R}$, $\omega_f(x)$ is either a $2$-cycle or $\{-\infty,\infty\}$.
\end{proposition}\medskip

\paragraph{\textit{The subsubcase $-1<\lambda<0$}} In this subsubcase, if $n\leqslant0$ then $$f(x)>\frac{1}{\lambda}\left(\left\lfloor-\frac{\mu}{\lambda-1}\right\rfloor-\mu+(-n+1)+1\right),\,\,\,\text{i.e.},\,\,\,f(x)\in I_{-n+1}\cup I_{-n}\cup I_{-n-1}\cup\cdots\cup I_2\cup I_1,$$
whereas if $n\geqslant1$ then
$$f(x)\leqslant\frac{1}{\lambda}\left(\left\lfloor-\frac{\mu}{\lambda-1}\right\rfloor-\mu+(-n+1)\right),\,\,\,\text{i.e.},\,\,\,f(x)\in I_{-n+1}\cup I_{-n+2}\cup I_{-n+3}\cup\cdots\cup I_{-1}\cup I_0.$$
It follows that for every $n\in\mathbb{N}$, the sets
$$I_{-n+1}\cup I_{-n+2}\cup I_{-n+3}\cup\cdots\cup I_{-1}\cup I_0\qquad\text{and}\qquad I_n\cup I_{n-1}\cup I_{n-2}\cup\cdots\cup I_2\cup I_1$$
are invariant under $f^2$. This leads us to the following proposition, which can be proved in the same way as Proposition \ref{prop7}.\medskip

\begin{proposition}\label{prop9}
Suppose $\left\lfloor-\frac{\mu-1}{\lambda-1}\right\rfloor+1>\left\lfloor-\frac{\mu}{\lambda-1}\right\rfloor$. If $-1<\lambda<0$, then for every $x\in\mathbb{R}$, $\omega_f(x)$ is a $2$-cycle.
\end{proposition}\medskip

\noindent This completes our proof of Theorem \ref{thm:omegalimit}.

\section*{Disclosure statement}

No potential conflict of interest was reported by the author.


\end{document}